\documentclass[letterpaper,12pt]{amsart}

\usepackage{amsmath,amsthm,amssymb,mathtools,setspace,dsfont,hyperref,color}
\renewcommand{\thispagestyle}[1]{} 
\renewcommand{\qed}{$\blacksquare$}
\newcommand{\dia}{$\diamond$}

\newtheorem{thm}{Theorem}[section]

\newtheorem{cor}[thm]{Corollary}
\newtheorem{lem}[thm]{Lemma}

\newtheorem{dfn}[thm]{Definition} 

\newtheorem{rem}[thm]{Remark}

\newtheorem{ex}[thm]{Example}

\newcommand{\eps}{\varepsilon}
\newcommand{\vol}{\mathrm{Vol}}

\newcommand{\N}{\mathbb{N}}

\newcommand{\R}{\mathbb{R}}

\newcommand{\bO}{\mathbf{O}}

\newcommand{\cL}{\mathcal{L}}

\newcommand{\abs}[1]{\left|#1\right|}
\newcommand{\norm}[1]{\left\lVert#1\right\rVert}

\usepackage{color}
\definecolor{red}{rgb}{.7,0,0}

\begin{document}
\title[Smoothed Analysis of Structured Real Polynomial 
Systems]{\mbox{}
\vspace{-1in}\\ 
Smoothed Analysis for the Condition Number of 
Structured Real Polynomial Systems}  

\author{Alperen A. Erg\"ur}
\address{ University of Texas at San Antonio, One UTSA Circle, San Antonio, TX, 78249, USA }
\email{alperen.ergur@utsa.edu}
\author{Grigoris Paouris}
\address{Department of Mathematics,
Texas A\&M University TAMU 3368,
College Station, Texas \ 77843-3368, USA.}
\email{grigoris@math.tamu.edu}
\author{J.\ Maurice Rojas}
\address{Department of Mathematics,
Texas A\&M University TAMU 3368,
College Station, Texas \ 77843-3368, USA.}
\email{rojas@math.tamu.edu}
\thanks{A.E.\ was partially supported by Einstein Foundation, Berlin and by 
Pravesh Kothari of CMU.  G.P.\ was partially supported by Simons Foundation 
Collaboration grant 527498 and NSF grants DMS-1812240 and CCF-1900881. 
J.M.R.\ was partially supported by NSF 
grants CCF-1409020, DMS-1460766, and CCF-1900881.}

\begin{abstract} 
We consider the sensitivity of real zeros of structured polynomial systems to 
pertubations of their coefficients. In particular, we provide explicit 
estimates for condition numbers of structured random real polynomial 
systems and extend these estimates to the smoothed analysis setting. 
\end{abstract}

\maketitle

\section{Introduction}
Efficiently finding real roots of real polynomial systems is one of the main 
objectives of computational algebraic geometry.  There are numerous algorithms 
for this task, but the core steps of these algorithms are easy to outline: 
They are some combination of algebraic manipulation, a discrete/polyhedral 
computation, and a numerical iterative scheme. 

From a computational complexity point of view, the cost of numerical iteration 
is much less transparent than the cost of algebraic or discrete computation. 
This paper constitutes a step toward understanding the complexity of 
numerically solving structured real polynomial systems. Our main results 
are Theorems \ref{intro1}, \ref{intro2}, and \ref{intro3} below but we will 
first need to give some context for our results. 

\subsection{How to control accuracy and complexity of numerics in real 
algebraic geometry?}
In the numerical linear algebra tradition, going back to von Neumann and 
Turing, condition numbers play a central role in the measurement of speed 
and the control of accuracy of algorithms (see, e.g., \cite{bcss,cond} for further 
background). Shub and Smale initiated the use of condition numbers for 
polynomial system solving over the field of complex numbers \cite{SS,Smale}. 
Subsequently, condition numbers played a central role in the solution of 
Smale's 17th problem \cite{pardo,cond,lairez}.  

The numerics of solving polynomial systems over the real numbers is more subtle 
than complex case: small perturbations can cause the solution set to change 
cardinality. One can even go from having no real zero to many real zeros by 
an arbitrarily small change in the coefficients.  This behaviour doesn't 
appear over the complex numbers as one has theorems (such as the Fundemantel 
Theorem of Algebra) proving that root counts are ``generically'' constant.  
Luckily, a condition number theory that captures these subtleties was 
developed by Cucker \cite{cucker}.  Now we set up the notation and present 
Cucker's definition.
\begin{dfn}[Bombieri-Weyl Norm]
We set $x^\alpha\!:=\!x^{\alpha_1}_1\cdots x^{\alpha_n}_n$ where $\alpha\!:=\!(\alpha_1,\ldots,\alpha_n)$, and let 
$P=(p_1,\ldots,p_{n-1})$ be a system of homogenous polynomials with degree pattern $d_1,\ldots,d_{n-1}$. 
Let $c_{i,\alpha}$  denote the coefficient of $x^\alpha$ in a $p_i$. 
We define the {\em Weyl-Bombieri norms} of $p_i$ and $P$ to be, respectively,
\[ \|p_i\|_W\!:=\!\sqrt{\sum\limits_{\alpha_1+\cdots+\alpha_n=d_i} \frac{|c_{i,\alpha}|^2}{\binom{d_i}{\alpha}}} \] 
and
\[ \|P\|_W\!:=\!\sqrt{\sum\limits^{n-1}_{i=1} \|p_i\|^2_W}. \ \text{\dia} 
\]    
\end{dfn}

The following is Cucker's condition number definition \cite{cucker}.
\begin{dfn}[Real Condition Number] 
For a system of homogenous polynomials $P=(p_1,\ldots,p_{n-1})$ with degree pattern $(d_1,\ldots,d_{n-1})$, let $\Delta_{n-1}$ be the diagonal matrix 
with entries $\sqrt{d_1},\ldots,\sqrt{d_{n-1}}$ and let
\[ DP(x)|_{T_x S^{n-1}} : T_x S^{n-1} \longrightarrow \R^m \] 
denote the linear map between tangent spaces induced by the Jacobian matrix of the polynomial system $P$ evaluated at the point 
$x \in S^{n-1}$. 

The local condition number of $P$ at a point $x \in S^{n-1}$ is  
\[ \tilde{\kappa}(P,x):= \frac{\norm{P}_{W}}{\sqrt{ \norm{ DP(x)|^{-1}_{T_x S^{n-1}} \Delta_{n-1} }^{-2}   +\norm{P(x)}_2^2}} \]
and the global condition number is 
\[ \tilde{\kappa}(P):=\sup\limits_{x \in S^{n-1}} \tilde{\kappa}(P,x). \ 
\text{\dia} \]
\end{dfn} 
An important feature of Cucker's real condition number is the following geometric fact\cite{M2}.
\begin{thm}[Real Condition Number Theorem]
We use $H_D$ to denote the vector space of homogenous polynomial systems with degree pattern $(d_1,\ldots,d_{n-1})$, and equip this space with the metric
$\rho(.,.)$ induced by the Bombieri-Weyl norm.  We define the set of 
{\em ill-posed problems} to be: 
\[ \Sigma:= \{ P \in H_D : P \; \text{has a singular zero in} \; S^{n-1} \} \]
Then we have 
\[ \tilde{\kappa}(P) = \frac{\norm{P}_W}{\rho(P,\Sigma)} \; . \] 
\end{thm}

Cucker's condition number is used in the design and analysis of a numerical algorithm for real zero counting \cite{M1,M2,M3}, in the series of papers for computing homology groups of semialgebraic sets \cite{cks,homosemi,homosemi2}, and 
more recently in the analysis of a well-known algorithm for meshing curves and 
surfaces (the Plantinga-Vegter algorithm) \cite{pv}. 

One important observation is that the complexity of a numerical algorithm 
over the real numbers (imagine using bisection for finding real zeros of a 
given univariate polynomial) varies depending on the geometry of the input, 
and not just the bit complexity of its vector representation. Therefore it is 
more natural to go beyond worst-case analysis and seek quantitative bounds 
for ``typical'' inputs. We now explain the existing attempts toward 
mathematically modeling the intuitive phrase ``typical input''. 

\medskip 
\noindent 
\textbf{Random and adverserial random models.}  
Worst-case complexity theory, spearheaded by the P vs. NP question, has been a driving force behing many algoritmic breaktrhorughs in the last five decades. 
However, it has become clear that worst-case complexity theory fails to 
capture the practical performance of algorithms. The unreasonable 
effectiveness of everyday statistical methods are a case in point: the spotify 
app on cell phones solves instances of an NP-Hard problem all the 
time!  

Two dominant paradigms for going beyond the worst-case analysis of algorithms are as follows:  Assume an algorithm $\mathrm{T}$ operates on the input 
$x\in \mathbb{R}^k$, with the cost of output $\mathrm{T}(x)$ bounded from above 
by $\mathrm{C}(x)$. One then equips the input space $\mathbb{R}^k$ with a 
probability measure $\mu$ and considers the average cost 
$\mathbb{E}_{  x \sim \mu}  \mathrm{C}(x)$,  
or {\em smoothed analysis of the cost with parameter $\delta > 0$}: 
$\sup_{x \in \mathbb{R}^k } \mathbb{E}_{  y  \sim \mu }  
\mathrm{C}(x + \delta \norm{x} y)$.  
Clearly, as $\delta \rightarrow 0$, smoothed complexity recovers worst-case 
complexity, and when $\delta\longrightarrow\infty$ we recover average-case 
complexity. It is also clear that to have a realistic complexity analysis, one 
should have a probability measure $\mu$ that somehow reflects one's 
context, and use theorems that allow a broad class of measures $\mu$. 
The idea of smoothed analysis originated in work of Spielman and Teng 
\cite{ST}.

\subsection{Existing results for average and smoothed analysis}
Existing results for the average analysis of real condition number from \cite{M3} can be roughly summarized as follows.
\begin{thm}[Cucker, Krick, Malajovich, Wschebor] 
Suppose 
\[ p_i(x):=\sum_{\alpha_1+\ldots+\alpha_n=d_i}  c^{(i)}_{\alpha} x^{\alpha} 
\; , \;  i=1,\ldots,n-1 \]
are random polynomials where $c^{(i)}_{\alpha}$ are centered Gaussian random variables with variances $\binom{d_i}{\alpha}$. Then, for the random polynomial 
system $P=(p_1,\ldots,p_{n-1})$ and for all $t \geq 1$, we have
\[ \mathbb{P}\left\{ \tilde{\kappa}(P) \geq  t \; 8 d^{\frac{n+4}{2}} 
	n^{5/2} N^{1/2} \right\} \leq  \frac{ (1+\log(t))^{\frac{1}{2}} }{t} \]
where $d=\max_{i} d_i$ and $N=\sum_{i=1}^{n-1} \binom{n+d_i-1}{d_i}$.  
\end{thm}
Recall the following smoothed analysis type result from \cite{M2}: 
\begin{thm}[Cucker, Krick, Malajovich, Wschebore]
 Let $Q$ be an arbitrary polynomial system with degree pattern 
$(d_1,\ldots,d_{n-1})$,  
let $P=(p_1,\ldots,p_{n-1})$ be a random polynomial  system  as defined 
above. Now for a parameter $0<\delta<1$ we define a random perturbation of $Q$ 
with $(P, \delta)$ as follows: $G :=Q+ \delta \norm{Q}_W P$.  
Then we have
\[ \mathbb{P} \left\{ \tilde{\kappa}(G) \geq t \frac{13 n^2  d^{2n+2} N}
{\delta} \right\} \leq  \frac{1}{t}  . \] 
\end{thm} 
\begin{rem}
The randomness model considered in these seminal results has the following 
restriction: the induced probability measure is invariant under the action of the orthogonal group $O(n)$ on the space of polynomials. The proof techniques 
used in the papers seems to be only applicable when one has this 
group invariance property. This creates an obstruction against anaylsis on 
spaces of structured polynomials; spaces of structured polynomials are not 
necessarily closed under the action of $O(n)$, and hence do not support an 
$O(n)$-invariant probability measure.  
\end{rem}
\subsection{What about structured polynomials?}
Let $H_{d_i}$ be the vector space of homogenous polynomials with $n$ variables, 
and let $H_D$ be the vector space of polynomial systems with degree pattern 
$D=(d_1,\ldots,d_{n-1})$. Let $E_i \subset H_{d_i}$ be linear subspaces 
for $i=1,\ldots,n-1$, and let $E=(E_1,\ldots,E_{n-1})$ be the 
corresponding vector space of polynomial systems.  

For virtually any application of real root finding algorithms, the user has a polynomial system with a particular structure rather than a generic polynomial system with $N=\sum_{i=1} \binom{n+d_i-1}{d_i}$ many coefficients. Suppose a user has identified the linear structure $E$ that is present in the target equations, and would like to know about how much precision is expected for 
round-off errors in the space $E$. One could  
induce a probability measure $\mu$ on $E$ and use 
$\mathbb{E}_{P \sim \mu} \log \left( \tilde{\kappa}(P) \right)$ to determine 
the expected precision of numerical solutions. What could go wrong?
\begin{ex} 
Let $u,v \in S^{n-1}$ be two vectors with $u \bot v$, and define the following 
subspaces: 
\[ E_i := \{ p \in H_{d_i} : p(u) = \langle \nabla p (u) , v \rangle = 0 \}  \; , \; i=1,\ldots,n-1 \]
where $\nabla p (u) $ denotes the the gradient of $p$ evaluated at $u$.  $E_i$ are codimension $2$ linear subspaces  of $H_{d_i}$. 
Now consider the space of polynomials  $E := (E_1, \ldots, E_{n-1})$; any polynomial system in the space $E$  has a singular real zero at $u$. 
Hence,  for all $P \in E$ the condition number $\tilde{\kappa}(P)$ is infinite.  
\end{ex}

The preceding example illustrates that, for certain linear spaces $E$, the 
probabilistic analysis of condition numbers is meaningless: It is possible 
for {\em certain} spaces $E$ that all inhabitants have infinite condition 
number. We will rule out these degenerate cases as follows.
\begin{dfn}[Non-degenerate linear space]
We call a linear space $E_i \subset H_{d_i}$ non-degenerate if for all $v \in S^{n-1}$, there exists an element $p_i \in E_i$ with $p_i(v) \neq 0$. In other words, $E_i$ is non-degenerate if there is no base point $v \in S^{n-1}$ where 
all the elements of $E_i$ vanish all together.
We call a space of polynomial systems $E=(E_1,\ldots,E_{n-1})$ non-degenerate if all $E_i$ are non-degenerate for $i=1,\ldots,n-1$. \dia 
\end{dfn}
An easy corollary of Theorem \ref{intro1} shows that the expected precision  
is finite for any non-degenerate space $E$. 
\begin{cor}
Let $E \subset H_D$ be a non-degenerate linear space of polynomials. Let $\mu$ 
be a probability measure supported on the space $E$ that satisfies the 
assumptions listed in section \ref{model}. Then 
$ \mathbb{E}_{P \sim \mu} \log \left( \tilde{\kappa}(P) \right)$  is finite.
\end{cor}
This is clearly not the end of the story: a non-degenerate linear structure $E$ may still be close to being degenerate, and this would make every element in the space $E$ ill-conditioned. So we need to somehow quantify the numerical 
conditioning of a linear structure $E$. Next, we introduce the notion of 
{\em dispersion} as a rough measure of conditioning of a linear structure.  
\subsection{The dispersion constant of a linear space} 
Suppose a linear subspace $F \subset H_{d}$ is given for some $d>1$ together with an orthonormal basis $u_{j}(x) \; , \; j=1,\ldots,m$ with respect to Bombieri-Weyl norm. Now suppose for a particular point $v_0 \in S^{n-1}$  all the basis elements satisfy $abs{u_{j}(v_0)} < \eps$ where $\eps>0$ is small. What kind of behavior one would expect from elements of $F$ at the point $v_0$? This point $v_0$ would behave like a base point (like if all elements of $F$ vanishes at $v_0$) unless one employs rather high precision.  This motivates the following 
definition.
\begin{dfn}[Dispersion constant of a linear space of polynomials]
Let $F \subset H_{d}$ is given for some $d>1$, and let $u_{j}(x) \; , \; j=1,\ldots,m$ be an orthonormal basis of $F$ with respect to Bombieri-Weyl norm. We define the following two quantities
\[ \sigma_{min}(F) := \min_{v \in S^{n-1}} \left( \sum_{j} u_j(v)^2 \right)^{\frac{1}{2}} \; , \; \sigma_{max}(F) := \max_{v \in S^{n-1}} \left( \sum_{j} u_j(v)^2 \right)^{\frac{1}{2}}  \]
and the dispersion constant $\sigma(F)$ is their ratio:
\[ \sigma(F) := \frac{\sigma_{max}(F)}{\sigma_{min}(F)}. \ \text{\dia} \]
\end{dfn}
The quantity $\sigma_{max}$ is introduced to make things scale invariant. We generalize the definition to polynomial systems in a straight-forward manner.
\begin{dfn}[Dispersion constant of a linear space of polynomial systems]
Let $E_i \subset H_{d_i}$ be linear spaces for $i=1,\ldots,n-1$, and let $E=(E_1,\ldots,E_{n-1})$. We define the dispersion constant $\sigma(E)$ as 
follows: $\sigma(E) := \max_{i} \sigma(E_i)$. \dia 
\end{dfn}

Our estimates replace the dimension $N$ in earlier results with the 
(potentially much smaller) dimension of $E$, at the expense of involving the 
new quantity $\sigma(E)$. So if a user has a fixed structure $E$ with small 
dimension and tame dispersion constant, then the expected conditioning on $E$ 
admits a much better bound than what earlier results suggest. On the other 
hand, if one has a sparse but highly sensitive structure, the resulting 
average-case conditioning could be a lot worse than the average over the 
entire space $H_D$.

\medskip 
\noindent  
\textbf{How big is the dispersion constant?} 
To better understand the dispersion constant, let us consider  two examples at opposite extremes. 
\begin{ex}[A subspace with minimal dispersion constant] 
 Consider subspaces of polynomials $F_i \subset P_{n,2d_i}$ defined as the span of
 \[ u_{kl}^{(i)}=(x_1^2+\cdots+x_n^2)^{d_i-1}x_kx_l \; \text{for} \;  1\leq k,l \leq n \]
 and let $F=(F_1,\ldots,F_n)$. It is easy to show that $\sigma(F)=1$. 
\end{ex}

\begin{ex}[A sparse but highly sensitive structure] 
Let $E \subset P_{n,d}$ be the subspace of polynomials spanned by the 
monomials $x_1^{d},\ldots,x_n^d$. Then, we have $\sigma(E)=n^{\frac{d-1}{2}}$.
\end{ex}

One may wonder how big the dispersion constant for a ``typical'' linear 
space $E$ is, for say, $E$ of dimension around $n^2 \log d$.  Would a typical 
low-dimensional space look like the second example or the first example? 
We address this question in the Appendix. For our main theorems, 
we will allow $E$ to be arbitrary and give bounds depending explicitly 
on the the dispersion constant $\sigma(E)$. 

\subsection{A general model of randomness for structured polynomial systems} 
\label{model}
In our precursor paper \cite{EPR17} we obtained probabilistic condition number 
estimates for general measures (without any group invariance assumption). In 
this paper we present probabilistic results for the same general family of 
measures, but this time supported on a structured space $E$ instead of $H_D$. 
Note that here the structured space $E$ will be fixed by the user, and our 
results will give estimates for a random element from $E$.  First, we introduce 
our general model of randomness.  

We say a random vector $X \in \mathbb{R}^n$ satisfies the 
{\em Centering}, {\em sub-Gaussian}, and {\em Small Ball} properties, with 
constants $K$ and $c_0$, if the following hold true: \\
\mbox{}\hspace{1cm}1.\ (Centering) For any  
$\theta \in S^{n-1}$ we have 
$ \mathbb{E}\langle X, \theta \rangle = 0$.\footnote{Equivalently, 
$\mathbb{E}  X \!=\!\bO$.}\\ 
\mbox{}\hspace{1cm}2.\ (Sub-Gaussian) There is a $K>0$ such that for every 
$\theta \in S^{n-1}$ we have \\ 
\mbox{}\hspace{4.5cm}$\mathrm{Prob} 
\left( \abs{ \langle X, \theta 
\rangle } \geq t \right) \leq 2 e^{-t^2/K^2}$ for all $ t>0$.\\
\mbox{}\hspace{1cm}3.\ (Small Ball) There is a $c_0>0$ such that for every 
vector $a \in \mathbb{R}^{n}$ we have\\ 
\mbox{}\hspace{3.8cm}$\mathrm{Prob}\left( \abs{ \langle a, X \rangle } \leq 
\eps \norm{a}_2 \right) \leq c_0 \eps$ for all $ \varepsilon>0$.  \\ 
We note that these three assumptions directly yield a relation between $K$ and $c_0$: We in fact have $Kc_0 \geq \frac{1}{4}$ (see \cite{EPR17} just before 
Section 3.2). Moreover,  for a random variable $X$ that satisfies above assumptions with constants $K$ and $c_o$, and a scalar $\lambda > 0$, the random varible $\lambda X$ satisfies the above assumptions with constants $\lambda K$ and $\lambda^{-1} c_0$. In other words $Kc_0$ is invariant under scaling, hence one can hope for a universal lower bound of $\frac{1}{4}$.

Random vectors that satisfy these three properties form a large family of 
distributions, including standard Gaussian vectors and uniform measures on a 
large family of convex bodies called $\Psi_2$-bodies (such as uniform measures 
on $l_p$-balls for all $p\geq2$). We refer the reader to the book of 
Vershynin \cite{Ve-1} for more details. Discrete sub-Gaussian distributions, 
such as the Bernoulli distribution, also satisfy an inequality similar to the 
small-ball inequality in our assumptions. However, 
the small-ball type inequality satisfied by such discrete distributions 
depends not only on the norm of the deterministic vector $a$ but also on the 
arithmetic structure of $a$. It is possible that our methods, combined with the 
work of Rudelson and Veshynin on the Littlewood-Offord problem 
\cite{RV}, can extend our main results to discrete distributions such as 
the Bernoulli distribution. In this work, we will content ourselves with 
continuous distributions. 

The preceding examples of random vectors do not necessarily 
have independent coordinates. This provides important extra flexibility. 
There are also interesting examples of random vectors with independent coordinates. In particular, if  $X_1,\ldots,X_m$ are independent centered random 
variables that each satisfy both the sub-Gaussian inequality with constant $K$ 
and the Small Ball condition with $c_0$, then the random vector 
$ X= (X_{1}, \ldots, X_{m})$ also satisfies the sub-Gaussian and Small Ball 
inequalities with constants $C_1K$ and  $C_{2} c_{0}$, where $C_{1}$ and $C_2$ 
are universal constants. This is a relatively new result of Rudelson and Vershynin \cite{RV-1}.  The best possible universal constant $C_2$ is 
discussed in \cite{LPP,PP}. To create a random variable satisfying the Small 
Ball and sub-Gaussian properties one can, for instance, start by fixing any 
$p \geq 2$ and then considering a random variable with density 
function $ f(t) :=c_{p} e^{-|t|^{p}}$, for suitably chosen positive $c_p$. 

\subsection{Our Results} We present estimates for random structured polynomial 
systems, where the randomness model is the one introduced in the preceding 
section. 

\medskip 
\noindent 
\textbf{Average-case condition number estimates for structured polynomial 
systems} 
\begin{thm} \label{intro1}
Let $E_i \subseteq H_{d_i}$ be non-degenerate linear subspaces, and let 
$E=(E_1,\ldots,E_{n-1})$.  Assume $ \dim(E)  \geq n \log(ed)$  and $n\geq 3$.
Let $p_i \in E_i$ be independent random elements of $E_i$ that satisfy the Centering property, the 
sub-Gaussian property with constant $K$, and the Small Ball property with 
constant $c_0$, each with respect to the Bombieri-Weyl inner product. We set 
$d := \max_{i} d_i$ and 
$M :=  n K \sqrt{\dim(E)} (c_0 C  K d^2 \log(ed)  \sigma(E))^{2n-2}$, 
where $C \geq 4$ is a universal constant. Then for the random polynomial system $P=(p_1, \ldots, p_{n-1})$, we have
\\  
\mbox{}\hfill $\mathrm{Prob}(\tilde{\kappa}(P) \geq t M) \leq  \begin{cases}
  3 t^{-\frac{1}{2}} &; \mbox{if } 1\leq t\leq e^{ 2 n \log{(ed)}} \\
  (e^2+1) t^{-\frac{1}{2} + \frac{1}{4\log(ed)}}  &; \mbox{if } e^{ 2n\log{(ed)} } \leq t  \\
   \end{cases}$\hfill\mbox{}\\ 
\noindent Moreover, for $0 < q < \frac{1}{2} - \frac{1}{2\log(ed)} $, we have $ \mathbb{E} (\tilde{\kappa}(P)^q)  \leq  M^q ( 1 + 4 q \log(ed) )$.   
In particular, $ \mathbb{E} \log(\tilde{\kappa}(P)) \leq 1+ \log M $. 	
\end{thm}

\medskip 
\noindent 
\textbf{Smoothed analysis for structured polynomial systems} 
For smoothed analysis we  need to introduce a slightly stronger assumption on the random input. 
This slightly stronger property is called the {\em Anti-Concentration Property} and it replaces the Small Ball assumption in our model of randomness. We will 
need a bit of terminology to define anti-concentration.
\begin{dfn} [Anti-Concentration Property]
For any real-valued random variable $Z$ and $t\!\geq\!0$, 
the {\em concentration function}, $F(Z,t)$,  is defined as 
$F(Z,t) := \max_{u \in \mathbb{R}} \mathrm{Prob} \{ |Z - u|  \leq t \}$.   
Let $\langle \cdot,\cdot \rangle$ denote the standard inner product 
on $\R^n$. 
We then say a random vector $X \in \mathbb{R}^n $ satisfies the 
{\em Anti-Concentration Property with constant $c_0$} if we have  
$  F( \langle X , \theta \rangle , \varepsilon)  \leq c_0 \varepsilon $ 
for all $\theta \in S^{n-1}$. \dia 
\end{dfn}

\noindent  It is easy to check that if the random variable $Z$ 
has bounded density $f$ 
then $F(Z, t) \leq \| f\|_{\infty} t$. Moreover, 
the Lebesgue Differentiation theorem states that 
upper bounds for the function $ t^{-1} F(Z, t) $ for all $t$ imply upper 
bounds for $ \| f\|_{\infty}$. See \cite{RV-2} for the details. 
\begin{thm} \label{intro2}
Let $E \subseteq H_D$ be a non-degenerate linear subspace for $D=(d_1,\ldots,d_{n-1})$. Assume  
$ \dim(E)  \geq n \log^2(ed)$ and $n\geq 3$. Let $Q \in E$ be a fixed 
(deterministic) polynomial system let $G \in E$ be a  random polynomial system given by the same model of randomness as in 
Theorem \ref{intro1}, but with the  Small Ball Property replaced by the Anti-Concentration Property. Set 
$d := \max_{i} d_i$,  and 
\[ M := n K \sqrt{\dim(E)} \left( c_0 d^2 C K \log(ed) \sigma(E) \right)^{2n-2} \left( 1 + \frac{\norm{Q}_{W}}{ \sqrt{n} K \log(ed)} \right)^{2n-1} \]  
where $C \geq 4$ is a universal constant. Then for the randomly perturbed polynomial system $P=Q+G$, we have 
\\  
\mbox{}\hfill $\mathrm{Prob}(\tilde{\kappa}(P) \geq t M) \leq  \begin{cases}
  3 t^{-\frac{1}{2}} &; \mbox{if } 1\leq t\leq e^{ 2 n \log{(ed)}} \\
  (e^2+1) t^{-\frac{1}{2} + \frac{1}{4\log(ed)}}  &; \mbox{if } e^{ 2n\log{(ed)} } \leq t  \\
   \end{cases}$\hfill\mbox{}\\ 
\noindent Moreover, for $0 < q < \frac{1}{2} - \frac{1}{2\log(ed)} $, we have 
\[ \mathbb{E} (\tilde{\kappa}(P)^q)  \leq  M^q ( 1 + 4 q \log(ed) ) . \] 
In particular, $\mathbb{E} \log(\tilde{\kappa}(P)) \leq 1+ \log M$. 
\end{thm}

We would like to consider a corollary to make the result easier to parse. 
\begin{cor}
Let $E \subseteq H_D$ be a non-degenerate linear subspace for $D=(d_1,\ldots,d_{n-1})$. Assume  
$ \dim(E)  \geq n \log(ed)^2$ and $n\geq 3$. Let $Q \in E$ be a fixed 
(deterministic) polynomial system, and let $G \in E$ be a  random polynomial system given by the model of randomness as in Theorem \ref{intro2}, but with 
fixed $K=1$. Now let $0<\delta<1$ be a parameter and consider the polynomial system 
\[ P:= Q + \delta \norm{Q}_W G \] 
We set $d := \max_{i} d_i$,  and 
\[ M := n  \sqrt{\dim(E)} \left( c_0 C d^2  \log(ed) \sigma(E) \right)^{2n-2} \delta \norm{Q}_W \left( 1 + \frac{1}{ \delta \sqrt{n} \log(ed)} \right)^{2n-1} \]  
where $C \geq 4$ is a universal constant. Then, we have
\\  
\mbox{}\hfill $\mathrm{Prob}(\tilde{\kappa}(P) \geq t M) \leq  \begin{cases}
  3 t^{-\frac{1}{2}} &; \mbox{if } 1\leq t\leq e^{ 2 n \log{(ed)}} \\
  (e^2+1) t^{-\frac{1}{2} + \frac{1}{4\log(ed)}}  &; \mbox{if } e^{ 2n\log{(ed)} } \leq t  \\
   \end{cases}$\hfill\mbox{}\\ 
\end{cor}

\medskip 
\noindent 
\textbf{An interesting consequence} As a corollary of the smoothed analysis estimate in Theorem \ref{intro2}, we derive the following structural result.
\begin{thm} \label{intro3}
Let $E_i \subseteq H_{d_i}$ be non-degenerate linear subspaces, let $E=(E_1,\ldots,E_{n-1})$, and let $Q \in E$. Then, for every $0 < \varepsilon < 1$, there is  
a polynomial system $P_{\varepsilon} \in E$ with the following properties: 
\[ \norm{P_{\varepsilon}-Q}_W \leq \varepsilon \norm{Q}_W  \left(\frac{\sqrt{\dim(E)}}{\log(ed)\sqrt{n}}\right) \]
and 
\[ \tilde{\kappa}(P_{\varepsilon}) \leq  \sqrt{n} \sqrt{\dim(E)}  \left( \frac{d^2 C \log(ed) \sigma(E)}{\varepsilon} \right)^{2n-2} \]
for a universal constant $C$.
\end{thm}
  One can view this result as a metric entropy statement as follows: Suppose we  are given a bounded set $\mathbb{T} \subset E$ with $\sup_{P \in T} 
\norm{P}_W \leq 1$, and we would like to cover $\mathbb{T}$ with balls of radius $\delta$, i.e., $\mathbb{T}=\bigcup_{i} B(p_i , \delta)$. Moreover, 
suppose we  want the ball-centers $p_i$ to have a controlled condition number. We can start with an arbitrary $\frac{\delta}{2}$ covering $\mathbb{T}=\bigcup_{i} 
B(p_i, \frac{\delta}{2})$, and use Theorem \ref{intro3} with  $\varepsilon=\frac{\delta\sqrt{n}}{2\sqrt{\dim(E)}}$ to find a $p_i$ with 
controlled condition number in each one of the balls  $B(p_i, \frac{\delta}{2})$. Then $\mathbb{T}=\bigcup_{i} B(p_i,\delta) $ gives 
a $\delta$-covering of $\mathbb{T}$ where $p_i$ has controlled condition  number. \\

\section{Background and Basic Estimates}
We first present a simple lemma for a single random polynomial.
\begin{lem}\label{basic}
Let $F \subset H_{d}$ be non-degenerate linear subspace of degree $d$ homogenous polynomials. We equip $F$ with Bombieri-Weyl norm. Suppose $p \in F$ is a random element 
that satisfies centering property, sub-Gaussian property with constant $K$, and small probability with constant $c_o$ each with respect to Bombieri inner product.
Then  for all $w \in S^{n-1}$ the  following estimates hold: 
$$ \mathrm{Prob}\left(  \abs{ p(w) } \geq t \sigma_{\mathrm{max}}(F)  \right) \leq \exp\left(1 - \frac{t^2}{K^2}\right)    $$
$$ \mathrm{Prob}\left(  \abs{ p(w) } \leq \varepsilon \sigma_{\mathrm{min}}(F) \right) \leq  c_0 \varepsilon. $$ 
\end{lem}
\begin{proof}
Suppose $u_1,\ldots,u_m$ is an orthonormal basis of $F$ with respect to Bombieri-Weyl inner product. Let $f \in F$ be a polynomial with $f(x)=\sum_i f_i u_i(x)$, then 
for any $v \in S^{n-1}$, clearly $f(v)=\sum_i a_i u_i(v)$. In other words, if we set $q_v:=\sum_{i} u_i(v) u_i(x)$ then we have $f(v)=\langle f , q_v \rangle_W$. Also note since $u_i$ is an orthonormal basis with respect to Bombieri norm, we have $\norm{q_v}_W=\left( \sum_i u_i(v)^2 \right)^{\frac{1}{2}}$. 

Now let  $p \in E'$ be the random element described above. The reasoning in the preceding paragraph gives us the following estimates for any fixed point $v \in S^{n-1}$:
$$ \mathrm{Prob}\left(  \abs{ p(v) } \geq t \norm{q_v}_W  \right) \leq \exp\left(1 - \frac{t^2}{K^2}\right)    $$
$$ \mathrm{Prob}\left(  \abs{ p(v) } \leq \varepsilon \norm{q_v}_W \right) \leq  c_0 \varepsilon. $$ 
By the definition of $\sigma_{\mathrm{max}}(F)$ and $\sigma_{\mathrm{min}}(F)$ these pointwise estimates yield the desired result.
\end{proof}

The following is the generalization of Lemma \ref{basic} to systems of polynomials.
\begin{lem} \label{systembasic}
Let $D=(d_1,\ldots,d_{n-1})\!\in\!\N^{n-1}$. For all $i\!\in\!\{1,
\ldots,n-1\}$ let $E_i \subseteq H_{d_i}$ be non-degenerate linear subspaces,  
and let $E:=(E_1,\ldots,E_{n-1})$. For each $i$, let $p_i$ be chosen from  
$E_i$ via a distribution satisfying the Centering 
Property, the Sub-Gaussian Property with constant $K$, and the Small Ball 
Property with constant $c_0$ (each with respect to the Bombieri-Weyl inner 
product). Then, for the random polynomial system $P=(p_1, \ldots,p_{n-1})$, and 
all $v \in S^{n-1}$, the following estimates hold: 
$$ \mathrm{Prob} \left( \norm{ P(v) }_2 \geq t \sigma_{\mathrm{max}}(E) 
\sqrt{n-1} \right) \leq \exp\left(1 - \frac{a_1 t^2 (n-1)}{K^2}\right)    $$ 
$$ \text{and \ \ \ } \mathrm{Prob} \left( \norm{ P(v) }_2 \leq \varepsilon \sigma_{\mathrm{min}}(E) \sqrt{n-1} \right) \leq  ( a_2 c_0 \varepsilon)^{n-1}, $$
\noindent where $a_1$ and $a_2$ are absolute constants. 
\end{lem}

For the proof of Lemma \ref{systembasic} we need to recall some theorems from probability theory and some basic tools developed in our earlier work
\cite{EPR17}. These basic lemmata will also be used throughout the paper. We start with a theorem which is  reminiscent of Hoeffding's 
classical inequality \cite{hoeffding}. 
\begin{thm} \cite[Prop.\ 5.10]{V} \label{Bernstein} 
There is an absolute constant $\tilde{c}_1 \!>\!0$ with the following 
property: If $X_1,\ldots, X_n$ are centered, sub-Gaussian random variables 
with  constant $K$, $a=(a_1,\ldots,a_n) \in \mathbb{R}^n$ 
and $t \geq 0$, then 
$$  \mathrm{Prob}\left( \left|\sum_i a_i X_i\right| \geq t  \right) \leq 2 
\exp\left(\frac{-\tilde{c}_1 t^2}{K^2 \norm{a}_2^2 }\right). \ \ \ \ 
\text{ \qed}$$ 
\end{thm}
We will also need the following standard lemma (see, e.g., \cite[Lemma 2.2]{RV}).
\begin{lem}
\label{smallballtensor} Assume $ Z_{1}, \ldots , Z_{n}$ are independent
random variables that have the property that 
$F(Z_{i}, t) ) \leq c_{0} t $ for all $t>0$. Then for $ t>0$ we have 
$ F ( W, t \sqrt{n}) \leq ( c c_{0} t)^{n}$,  
where $ W:= \| ( Z_{1}, \ldots, Z_{n})\|_{2} $. Moreover, 
if $ \xi_{1}, \ldots , \xi_{k}$ are independent random variables 
such that, for every $\eps>0$, we have 
$\mathrm{Prob} \left ( | \xi_{i} | \leq \varepsilon \right) \leq c_0 
\varepsilon$. Then there is a universal constant $\tilde{c}>0$ such that for 
every $\varepsilon >0$ we have $\mathrm{Prob} 
\left( \sqrt{ \xi_{1}^{2} + \cdots + \xi_{k}^{2} } \leq \varepsilon \sqrt{k} 
\right) \leq \left( \tilde{c} c_0 \varepsilon \right)^{k}$. \qed 
\end{lem}

Now that we have our basic probabilistic tools we proceed to deriving 
some deterministic inequalities. 

The lemma below was proved in our earlier paper \cite{EPR17}, 
generalizing a classical Theorem of Kellog \cite{kellog}. To state the lemma 
we need a bit of terminology: For any system of homogeneous polynomials 
$P:=(p_1,\ldots , p_{n-1})\!\in\!(\R[x_1,\ldots,x_n])^{n-1}$ define 
$ \norm{P}_{\infty}:= \sup_{x \in S^{n-1}} \sqrt{\sum_{i=1}^{n-1} p_i(x)^2}$. 
Let $DP(x)$ denote the Jacobian matrix of the polynomial system at point $x$, let $DP(x)(u)$ denote the image of the vector $u$ under the linear operator 
$DP(x)$, and set 
$\left\|D^{(1)}P\right\|_{\infty}:= \sup_{x,u\in S^{n-1}} 
\norm{D P(x)(u)}_2$. (Alternatively, the last quantity can be written 
$\sup_{x,u \in S^{n-1}} \sqrt{\sum_{i=1}^{n-1} \langle 
\nabla p_i(x) , u \rangle^2}$.)   

\begin{lem} \label{kellogsys}
Let $P:=(p_1,\ldots,p_{n-1})\!\in\!(\R[x_1,\ldots,x_n])^{n-1}$ be a polynomial system 
with $p_i$ homogeneous of degree $d_i$ for each $i$ and set 
$d\!:=\!\max_i d_i$. Then:  
\begin{enumerate}
\item We have $\norm{D^{(1)}P}_{\infty} \leq d^2 \norm{P}_{\infty}$ and, 
for any mutually orthogonal $x,y\!\in\!S^{n-1}$, we also have  
$\norm{D P(x)(y)}_2  \leq d \norm{P}_{\infty}$. 
\item If $\deg(p_i)=d$ for all $i\in\{1,\ldots,n-1 \}$ then we also 
have $\norm{D^{(1)}P}_{\infty} \leq d \norm{P}_{\infty}$. \qed 
\end{enumerate}
\end{lem}

The final lemma we need is a discretization tool for homogenous polynomial systems that was developed in \cite{EPR17} based on Lemma \ref{kellogsys}. We need a bit of terminology to state the lemma. 

\begin{dfn}
Let $K$ be a compact set in a metric space $(X,d)$, then a set 
$A \subseteq K$ with finitely many elements is called a $\delta$-net if for 
every $x \in K$ there exists $y \in A$ with $d(x,y) \leq \delta$.  \dia 
\end{dfn}
For the unit sphere in $\mathbb{R}^n$, equipped with the standard Euclidean 
metric, there are known bounds for the size of a $\delta$-net. 
We recall one such bound below. 
\begin{lem} \label{netsize}
Let $S^{n-1}$ be the unit sphere in $\mathbb{R}^n$ with respect to standard euclidean metric. Then for every $\delta>0$, there exist a $\delta$-net $\mathcal{N} \subset S^{n-1}$ with size at most $2n(1+\frac{2}{\delta})^{n-1}$.
\end{lem}

\noindent 
Lemma \ref{netsize} is almost folklore: a proof appears in 
Proposition 2.1 of \cite{RV-2}. 

\begin{lem} \label{net-norm}
Let $P=(p_1,\ldots,p_{n-1})$ be a system of homogenous polynomials $p_i$ with $n$ variables and $\deg(p_i)=d_i$. Let $\mathcal{N}$ be a $\delta$-net on $S^{n-1}$. Let $\max_{\mathcal{N}}(P)= \sup_{y \in \mathcal{N}}  \norm{P(y)}_2 $ and $\norm{P}_{\infty}=\sup_{x \in S^{n-1}} \norm{P(x)}_2$. Similarly let us define,
 $$\max_{\mathcal{N}^{k+1}}(D^{(k)}P) =\sup_{x,u_1,\ldots,u_k \in \mathcal{N}} \norm{D^{(k)} P(x) (u_1, \ldots, u_k)}_2 $$ 

$$\text{and \ \ }  \norm{D^{(k)} P}_{\infty}=\sup_{x,u_1,\ldots,u_k \in S^{n-1}} \norm{D^{(k)} P(x) (u_1,\ldots, u_k)}_2. $$

\noindent 
Then 
\begin{enumerate}
\item When $\deg(p_i)=d$ for all $ i \in \{1,\ldots,m \}$ we have 
$ \norm{P}_{\infty} \leq \frac{\max_{\mathcal{N}}(P)}{1-d\delta} $  
and\\  
$ \norm{D^{(k)} P}_{\infty} \leq \frac{\max_{\mathcal{N}^{k+1}}(D^{(k)}P)}{1-\delta d \sqrt{k+1}} $. 

\item \scalebox{.95}[1]{When $ \max_{i} \{ deg(p_i) \} \leq d$ we have 
$\norm{P}_{\infty} \leq \frac{\max_{\mathcal{N}}(P)}{1-d^2\delta}$ and 
$\norm{D^{(k)} P}_{\infty} \leq \frac{\max_{\mathcal{N}^{k+1}}(D^{(k)} P) }{1-\delta d^2 \sqrt{k+1}}$. \qed}  
\end{enumerate}
\end{lem}

\noindent 
{\bf Proof of Lemma \ref{systembasic}:} 
We begin with the first claim. Using Lemma \ref{basic} and the fact that 
$\sigma_{\mathrm{max}}(E) \geq \sigma_{\mathrm{max}}(E_i)$ for all $i$, 
we get the following estimate for any $p_i \in E_i$ and $w \in S^{n-1}$:
$$ \mathrm{Prob}\left(  \abs{ p_i(w) } 
\geq s \sigma_{\mathrm{max}}(E) \right) \leq 
\exp\left(1-\frac{s^2}{K^2}\right). $$
\noindent Now let $a=(a_1,\ldots,a_{n-1}) \in \mathbb{R}^{n-1}$ with $\norm{a}_2=1$, and apply Lemma \ref{Bernstein} to the sub-Gaussian random variables 
$ \frac{ p_i(w) }{\sigma_{\mathrm{max}}(E)}$ and the vector $a$. We then get 
$$ \mathrm{Prob}\left(  \left|\sum_i a_i p_i(w)\right|  \geq s \sigma_{\mathrm{max}}(E) \right) \leq \exp\left(1 - \frac{\tilde{c}_1 s^2}{K^2}\right).$$  
\noindent 
Observe that $\norm{P(w)}_2=\max_{a \in S^{n-2}} \abs{ \langle a , 
P(w)  \rangle }$. For any fixed point $w \in S^{n-1}$ and a free 
variable $a \in \mathbb{R}^n$, we have that $\langle a , P(w)  \rangle$ is a 
linear polynomial on $a$. We then use Lemma \ref{net-norm} on this linear 
polynomial, which gives us the following estimate:
$$ \mathrm{Prob}\left(  \norm{P(w)}_2  \geq  \frac{s \sigma_{\mathrm{max}}(E)}{1-\delta}  \right)  \leq \abs{\mathcal{N}} 
\exp\left(1 - \frac{\tilde{c}_1 s^2}{K^2} \right). $$
\noindent 
We then use Lemma \ref{netsize} to control the cardinality of the 
$\delta$-net and get  
\[ \abs{\mathcal{N}} \leq 2n(1+\frac{2}{\delta})^{n-1} \leq e^{(n-1)\tilde{c}\log(\frac{1}{\delta})}, \]
for some absolute constant $\tilde{c}$. 
So we set $t=2 s \sqrt{n-1}$, 
$\delta=\frac{1}{2}$, and obtain the following estimate for some 
universal constant $a_1$. 
$$ \mathrm{Prob}\left(  \norm{ P(w) }_2 \geq t \sigma_{\mathrm{max}}(E) \sqrt{n-1} \right) \leq \exp\left(1 - \frac{a_1 t^2 (n-1)}{K^2}\right).  $$
\noindent We continue with the proof of the second claim. Using Lemma 
\ref{basic} and the fact that \mbox{$\sigma_{\mathrm{min}}(E) \leq 
\sigma_{\mathrm{min}}(E_i)$} for all $i$, we deduce the following estimate 
for all $p_i$ and for any $\varepsilon > 0$:
$$ \mathrm{Prob}\left(  \abs{  \frac{p_i(w)}{\sigma_{\mathrm{min}}(E)} } \leq  \varepsilon  \right) \leq c_0 \varepsilon.  $$
\noindent  Using Lemma \ref{smallballtensor} on the random variables 
$\abs{  \frac{p_i(w)}{\sigma_{\mathrm{min}}(E)} }$ gives the following 
estimate:  
$$ \mathrm{Prob}\left(  \norm{P(w)}_2  \leq  \varepsilon \sigma_{\mathrm{min}} (E) \sqrt{n-1}  \right) \leq ( \tilde{c}_2 c_0 \varepsilon )^{n-1}. 
\text{ \qed}$$ 

\section{Operator Norm Type Estimates}
In this section we will estimate the absolute maximum norm of a random 
polynomial system on the sphere. Recall that for a homogenous polynomial system $P=(p_1,\ldots, p_{n-1})$ the sup-norm is defined as 
$\norm{P}_{\infty}=\sup_{x \in S^{n-1}} \norm{P(x)}_2$. The following lemma is 
our sup-norm estimate for a random polynomial system $P$.

\begin{lem} \label{operatornorm}
Let $D=(d_1,\ldots,d_{n-1})$ be a vector with positive integer coordinates, 
let $E_i \subseteq H_{d_i}$ be full linear subspaces, and let $E=(E_1,\ldots,E_{n-1})$. 
Let $p_i \in E_i$ be independent random elements of $E_i$ that satisfy the 
Centering Property, the Sub-Gaussian Property with constant $K$, and the 
Small Ball Property with constant $c_0$, each with 
respect to Bombieri-Weyl inner product. Let $\mathcal{N}$ be a $\delta$-net on 
$S^{n-1}$. Then for $P=(p_1,\ldots,p_{n-1})$ we have
$$ \mathrm{Prob} \left( \max_{x \in \mathcal{N}} \norm{P(x)}_2 \geq t \sigma_{\mathrm{max}}(E) \sqrt{n} \right) \leq
\abs{\mathcal{N}} \exp\left(1-\frac{a_1t^2n}{K^2}\right),$$
where $a_1$ is a universal constant. In particular, for $d=\max_{i} \deg(p_i)$, 
$\delta=\frac{1}{3d^2}$, and  $t=s\log(ed)$ with $s \geq 1$ this gives us the 
following estimate
$$ \mathrm{Prob} \left( \norm{P}_{\infty} \geq s \sigma_{\mathrm{max}}(E) \sqrt{n} \log(ed) \right) \leq \exp\left(1- \frac{a_3s^2 n \log(ed)^2}{K^2}\right)$$
where $a_3$ is a universal constant. 
\end{lem}

\begin{proof}
The first statement is proven by just taking a union bound over $\mathcal{N}$ and using Lemma \ref{systembasic}. The second part of the statement immediately follows by using the first part and Lemma \ref{net-norm}.
\end{proof}

\section{Small Ball Type Estimates}
We define the following quantity for later convenience. 
$$ \cL(x,y):= \sqrt{\norm{\Delta^{-1}_m D^{(1)}P(x)(y)}_2^2+ \norm{P(x)}_2^2}$$ 
\noindent It follows directly that 
\[ \frac{\norm{P}_W}{\tilde{\kappa}(P,x)}= \sqrt{\|P\|^2_W \tilde{\mu}_\mathrm{norm}(P,x)^{-2} +\norm{P(x)}_2^2}=\inf\limits_{\substack{y \bot x \\ y\in S^{n-1}}} \cL(x,y) \]    
So we set $L(P,x)=\frac{\norm{P}_W}{\tilde{\kappa}(P,x)}$ and  $L(P)=\min_{x \in S^{n-1}} L(P,x)$. 
We then have the following equalities: 
\[ L(P,x)=\inf\limits_{\substack{y \bot x \\ y\in S^{n-1}}} \cL(x,y) \; , \;  \tilde{\kappa}(P,x)= \frac{\norm{P}_W}{L(P,x)} \]
and, finally
\[ \tilde{\kappa}(P)=\frac{\norm{P}_W}{L(P)} .\]
In this section, we prove a small-ball type estimate to control behavior 
of the denominator $L(p)$. We first need to recall a technical lemma from our 
earlier paper \cite{EPR17}, which builds on an idea of Nguyen \cite{Ng}. 
\begin{lem}\label{Taylor}
Let $n\geq 2$, let $P:=(p_1,\ldots,p_{n-1})$ be a system of $n$-variate homogenous 
polynomials, and assume $\norm{P}_{\infty} \leq \gamma$. Let 
$x,y \in S^{n-1}$ be mutually orthogonal vectors with $\cL(x,y) \leq \alpha$, 
and let $r\in[-1,1]$. 
Then for every $w$ with $w=x+ \beta r y + \beta^2 z$ for some $z \in B_2^{n}$, 
we have the following inequalities:   
\begin{enumerate}
\item If $d:=\max_i d_i$ and $ 0 < \beta \leq d^{-4}$ then 
$\norm{P(w)}_2^2 \leq 8 ( \alpha^2 + (2+e^4) \beta^{4} d^4 \gamma^2)$. 
\item \scalebox{.92}[1]{If $\deg (p_i)=d$ for all $i \in [n-1]$ and 
$ 0 < \beta \leq d^{-2}$  then $\norm{P(w)}_2^2 \leq 8 ( \alpha^2 
+ (2+e^4) \beta^{4} d^4 \gamma^2 )$. \qed}  
\end{enumerate}
\end{lem}

We also need to state and prove the following simple Lemma for the clarity of succeeding proofs.  
\begin{lem} \label{Mcheck}
Let $n \geq 1$ be an integer. Then for 
$0 \leq x \leq \frac{1}{n}$ we have $(1+x)^n \leq 1+ 3nx$.  
\end{lem}
\begin{proof}
 For every $0 \leq y \leq 1$ we have $1+3y \geq e^y$. This can be seen by 
setting $f(y)=1+3y-e^y$, observing $f^{'}(y)>0$ for all $0 \leq y \leq 1$ 
and $f(1) >0$, $f(0)=0$. With a similar reasoning one can prove 
$e^x \geq 1+x$, 
and hence $e^{nx} \geq (1+x)^n$ for all $0 \leq x \leq 1$. Using $y=nx$ 
completes the proof.   
\end{proof}

\begin{thm}\label{L-theorem}
Let $D=(d_1,\ldots,d_{n-1})$ be a vector with positive integer coordinates, let $E_i \subseteq H_{d_i}$ be full linear subspaces, and let $E=(E_1,\ldots,E_{n-1})$. Let $p_i \in E_i$ be independent random elements of $E_i$ that satisfy 
the Centering Property, the Sub-Gaussian Property with constant $K$, and 
the Small Ball Property with constant $c_0$, each with respect to Bombieri-Weyl inner product. Let $\gamma \geq 1 $, $d := \max_{i} d_i$, and assume $\alpha  \leq \min\{ d^{-8} , n^{-1} \} $. Then for $P=(p_1,\ldots,p_{n-1})$ we have 
$$ \mathrm{Prob} ( L(P)\leq \alpha ) \leq  \mathrm{Prob} \left( \|P\|_{\infty} 
\geq \gamma\right)  + c \alpha^{\frac{1}{2}} \sqrt{n} \left( \frac{c_0 d^2 
\gamma C}{ \sigma_{\mathrm{min}}(E) \sqrt{n}}\right)^{n-1} $$ 
\noindent where $C$ is a universal constant. 
\end{thm}

The proof of Theorem \ref{L-theorem} is similar to a proof in our earlier 
paper \cite{EPR17}. We reproduce the proof here due to the importance of 
Theorem \ref{L-theorem} in the flow of our current paper. 

\begin{proof} \noindent We assume the hypotheses of Assertion (1) in Lemma \ref{Taylor}: Let $ \alpha,\gamma >0$ and 
$\beta \leq d^{-4}$. Let $ {\bf B:}= \{P\; | \;  \|P\|_{\infty} \leq 
\gamma\}$ and let \\ 
\mbox{}\hfill $ {\bf L} := \{P\; | \; L(P) \leq \alpha\} = 
\{ P\; | \; \text{There exist } x,y\!\in\!S^{n-1} \text{ with }  
x\perp y \text{ and }  \cL(x,y) \leq  \alpha \} $.\hfill\mbox{}\\ 
Let $\Gamma:=8(\alpha^{2} + (2+e^4) \beta^{4} d^{4} \gamma^{2})$ and 
let $B^n_2$ denote the unit $\ell_2$-ball in $\R^n$. 

Lemma \ref{Taylor} implies that if the event ${\bf B}\cap {\bf L}$ occurs then there exists a non-empty
set $$V_{x,y}:= \{ w\in \mathbb R^{n} :w= x+ \beta r y+ \beta^2 z , x \perp y, \abs{r} \leq 1, z \perp y , z \in B_2^{n} \} \setminus B_{2}^{n} $$
such that $ \| P(w) \|_{2}^{2} \leq \Gamma$ for every $w$ in this set. 
Let $V:= \mathrm{Vol}\! \left( V_{x,y} \right)$. Note that for $ w\in V_{x,y}$ 
we have $\norm{w}_2^2 = \norm{x+\beta^2z}_2^2  + \norm{\beta y}_2^2 \leq 1 + 
4 \beta^2$. Hence we have $\norm{w}_2 \leq 1 + 2 \beta^2$. Since\linebreak 
$V_{x,y} \subseteq (1+ 2\beta^2 ) B_{2}^{n} 
\setminus B_{2}^{n} $, we have showed that\\ 
\mbox{}\hfill ${\bf B}\cap {\bf L}\subseteq 
\left\{P\; | \; \mathrm{Vol}\!\left(\{x\in (1+ 2 \beta^2 ) B_{2}^{n}\setminus 
B_{2}^{n} \; | \; \|P(x) \|_{2}^2 \leq \Gamma \}\right)  \geq V \right\}$.\hfill
\mbox{} \\  
Using Markov's Inequality, Fubini's Theorem, and Lemma \ref{systembasic}, we 
can estimate the probability of this event. Indeed, \\ 
$\mathrm{Prob}\left(  \mathrm{Vol}\! \left( \{ x\in (1+ 2 \beta^2) B_{2}^{n} \setminus 
B_{2}^{n}  : \|P(x) \|_{2}^2 \leq \Gamma \}  \right)  \geq V  \right)$\hfill\mbox{} 
\begin{eqnarray*} 
& \leq & \frac{1}{ V}  \mathbb{E} \mathrm{Vol}\!\left( \{ x\in (1+ 2 \beta^2) B_{2}^{n} \setminus B_{2}^{n}  : \|P(x) \|_{2}^{2} \leq \Gamma \}  \right) \\ 
& \leq & \frac{1}{V} \int_{ (1+ 2\beta^2) B_{2}^{n} \setminus B_{2}^{n} } 
\mathrm{Prob} \left(\|P(x)\|_{2}^{2} \leq \Gamma \right) dx \\ 
& \leq & \frac{ \mathrm{Vol}\!\left( (1+ 2\beta^2) B_{2}^{n} \setminus B_{2}^{n} \right) }{V} 
\max_{x\in (1+ 2\beta^2) B_{2}^{n} \setminus B_{2}^{n} }  \mathrm{Prob} 
\left( \| P(x)  \|_{2}^{2} \leq \Gamma \right).   
\end{eqnarray*} 
 
Now recall that $\vol(B_{2}^{n})= \frac{\pi^{n/2}}{ \Gamma\left(\frac{n}{2}+1
\right)}$. Then $\frac{\vol(B_{2}^{n})}{\vol(B_{2}^{n-1})} \leq \frac{c'}
{\sqrt{n}}$ for some constant $c'>0$.  If we assume that  
that $\beta^2 \leq \frac{1}{2n}$, then Lemma \ref{Mcheck} implies 
$(1+2\beta^2)^n \leq 1 + 6n \beta^2 $, and 
we obtain 
$$  \frac{\vol((1+ 2\beta^2) B_{2}^{n} \setminus B_{2}^{n}) }{V} 
\leq \frac{\vol(B_{2}^{n})\left( (1+ 2\beta^2)^{n} -1\right) }{ \beta 
(\beta^2)^{n-1} \vol(B_{2}^{n-1})} \leq c \sqrt{n} \beta \beta^{2-2n} ,$$
for some absolute constant $c>0$. Note that here, for a lower bound on $V$, we 
used the fact that $V_{x,y}$ contains more than half of a cylinder with 
base having radius $\beta^2$ and height $2\beta$. 

Writing $ \tilde{x} := \frac{ x}{ \|x\|_{2}}$
for any $ x\neq 0$ we then obtain, for $z\notin B_{2}^{n}$, that 
$$ \| P (z) \|_{2}^{2} = \sum_{j=1}^{m} | p_{j} (z) |^{2} = \sum_{j=1}^{m} 
|p_j (\tilde{z}) |^{2} \| z\|_{2}^{2 d_j} \geq  \sum_{j=1}^{m} | p_{j} 
(\tilde{z}) |^{2} = \| P ( \tilde{z}) \|_{2}^{2}.$$ 
\noindent This implies, via Lemma \ref{systembasic}, that for every 
$w\!\in\!(1+ 2\beta^2) B_{2}^{n} \setminus 
B_{2}^{n}$ we have  
$$ \mathrm{Prob} \left( \|P(w) \|_{2}^{2} \leq \Gamma\right) \leq  \mathrm{Prob} \left( \|P(\tilde{w}) \|_{2}^{2} \leq \Gamma\right) \leq 
\left( c c_0 \sqrt{\frac{\Gamma}{ n \sigma_{\mathrm{min}}(E)^2}}\right)^{n-1}. $$

So we conclude that
$ \mathrm{Prob} \left({\bf  B} \cap {\bf L}\right) \leq  c \sqrt{n} 
\beta^{3-2n} \left( c c_0 \sqrt{\frac{\Gamma}{ n \sigma_{\mathrm{min}}(E)^2}} 
\right)^{n-1}$. Since 
$  \mathrm{Prob} \left( L(P) \leq \alpha \right) \leq \mathrm{Prob} 
\left( \norm{P}_{\infty} \geq \gamma \right) + \mathrm{Prob} ({\bf B}\cap 
{\bf L}) $ we then have

$$   \mathrm{Prob} \left( L(P) \leq \alpha \right) \leq  \mathrm{Prob} \left( \|P \|_{\infty} \geq \gamma\right)  + c \sqrt{n} \beta^{3-2n} \left( c c_0 \sqrt{\frac{\Gamma}{ n \sigma_{\mathrm{min}}(E)^2}} \right)^{n-1} $$
  
\noindent Recall that $ \Gamma= 8(\alpha^{2} + (5+e^4)\beta^{4} d^{4} \gamma^{2})$. 
We set $ \beta^2 := \alpha$. Our choice of $\beta$ and the assumption that $\gamma \geq 1$ then imply that 
$\Gamma \leq  C \alpha^2 \gamma^2 d^4$ for some constant $C$. So we obtain  
$$ \mathrm{Prob} ( L(P)\leq \alpha ) \leq  \mathrm{Prob} \left( \|P\|_{\infty} 
\geq \gamma\right)  + c \sqrt{n} (\alpha)^{\frac{3}{2}-n} \left( \frac{c_0 C  \alpha d^2 \gamma }{ \sigma_{\mathrm{min}}(E) \sqrt{n}}\right)^{n-1} $$ 

$$ \mathrm{Prob} ( L(P)\leq \alpha ) \leq  \mathrm{Prob} \left( \|P\|_{\infty} 
\geq \gamma\right)  + c \sqrt{n} (\alpha)^{\frac{1}{2}} \left( \frac{c_0 
d^2 \gamma C }{ \sigma_{\mathrm{min}}(E) \sqrt{n}}\right)^{n-1} $$ 

\noindent and our proof is complete. \end{proof}

\section{Proof of Theorem \ref{intro1}}
We first need to estimate Bombieri norm of a random polynomial system. The following lemma is more or less standard,  and it follows from Lemma \ref{Bernstein}.
\begin{lem} \label{AG}
Let $D=(d_1,\ldots,d_{n-1})$ be a vector with positive integer coordinates, let $E_i \subseteq H_{d_i}$ 
be full linear subspaces, and let $E=(E_1,\ldots,E_{n-1})$. Let $p_i \in E_i$ 
be random elements of $E_i$ that satisfy the Centering Property and the 
Sub-Gaussian Property with constant $K$, each with respect to Bombieri-Weyl 
inner product. Then for all $t\geq 1$, we have 
$$ \mathrm{Prob}\left(  \|p_{i}\|_W \geq t  \sqrt{\dim(E_i)}  \right) \leq 
\exp\left( 1 -\frac{t^{2} \dim(E_i)}{K^2}\right) $$ 
and  for the random polynomial system $P=(p_1,\ldots,p_{n-1})$ we have
$$ \mathrm{Prob}\left(  \|P\|_W \geq t  \sqrt{\dim(E)}  \right) \leq \exp
\left( 1 -\frac{t^{2} \dim(E)}{K^2}\right). \text{ \qed} $$  
\end{lem}
Now we have all the necessary tools to prove our probabilistic condition number theorem.  We will prove the following statement:
\begin{thm} \label{condition}
Let $D=(d_1,\ldots,d_{n-1})$ be a vector with positive integer coordinates, let $E_i \subseteq H_{d_i}$ be non-degenerate linear subspaces, and let $E=(E_1,\ldots,E_{n-1})$.
We assume that $ \dim(E)  \geq n \log(ed)$  and $n\geq 3$. Let $p_i \in E_i$ be independent random elements of $E_i$ that satisfy the Centering 
Property, the Sub-Gaussian Property with constant $K$, and the Small Ball Property with constant $c_0$, each with respect to the Bombieri-Weyl inner 
product. We  set $d := \max_{i} d_i$, and  
$$M :=  n K \sqrt{\dim(E)} (c_0 d^2 C K \log(ed)^2 \sigma(E))^{2n-2}  $$
\noindent where $C \geq 4$ is a universal constant. Then for $P=(p_1,\ldots,p_{n-1})$, we have\\  
\mbox{}\hfill $\mathrm{Prob}(\tilde{\kappa}(P) \geq t M) \leq  \begin{cases}
  \frac{3}{\sqrt{t}} &; \mbox{if } 1\leq t\leq e^{ 2 n \log{(ed)}} \\
  \frac{e^2+1}{\sqrt{t}} \left( \frac{{\log{t}}}
    {2 n \log{(ed)}}\right)^{\frac{n}{2}}  &; \mbox{if } e^{ 2n\log{(ed)} } \leq t  \\
   \end{cases}$\hfill\mbox{}\\ 
 \end{thm}

For notational simplictiy we set $m=\dim(E)$. To start the proof we observe the following:
$$  \mathrm{Prob}\left( \tilde{\kappa}(P) \geq t M \right) \leq  \mathrm{Prob}\left( \norm{P}_W \geq s K \sqrt{m} \right) +  \mathrm{Prob} \left( L(P) \leq  \frac{s K \sqrt{m}}{tM} \right)   $$

\noindent The first probability on the right hand side will be controlled by  Lemma \ref{AG}, and the second will be controlled by Theorem \ref{L-theorem}. Theorem \ref{L-theorem} states that for any $\gamma \geq 1$ and for\linebreak 
$\frac{s K \sqrt{m}}{tM} \leq \min \{ d^{-8} , n^{-1} \}$, we have 

$$ \mathrm{Prob} \left( L(P) \leq  \frac{sK\sqrt{m}}{tM} \right)  \leq \mathrm{Prob} \left( \norm{P}_{\infty} \geq \gamma \right) + \left(\frac{s K 
\sqrt{m}}{tM}\right)^{\frac{1}{2}}  \sqrt{n} \left(\frac{c_0 C \gamma d^2}{\sigma_{\mathrm{min}}(E) \sqrt{n}} \right)^{n-1}   $$

\noindent To have $\frac{s K \sqrt{m}}{tM} \leq \min \{ d^{-8} , n^{-1} \}$ is equivalent to $ t M \min \{ d^{-8} , n^{-1} \} \geq s K \sqrt{m} $. We will check this condition at the end of the proof. Now, for $\gamma= u \sigma_{\mathrm{max}}(E) \sqrt{n} \log(ed) K$ with $u \geq 1$, from Lemma \ref{operatornorm} we have 
$ \mathrm{Prob} \left( \norm{P}_{\infty} \geq  u \sigma_{\mathrm{max}}(E) \sqrt{n} \log(ed) K \right) \leq \exp(1- a_3u^2 n \log(ed)^2)$.  
That is, for $\gamma= u \sigma_{\mathrm{max}}(E) \sqrt{n} \log(ed) K$, 
we have the following estimate:\\ 
\scalebox{.93}[1]{
\vbox{$$ \mathrm{Prob} \left( L(P) \leq  \frac{s K \sqrt{m}}{tM} \right)  \leq 
\exp(1- a_3u^2 n \log(ed)^2) + \left(\frac{s K \sqrt{m}}{tM}\right)^{\frac{1}{2}} 
 \sqrt{n} \left(\frac{c_0 C u \sigma_{\mathrm{max}}(E) \log(ed) d^2 K}{\sigma_{\mathrm{min}}(E)} 
\right)^{n-1}.$$}} 

\noindent Since $\sigma(E)=\frac{\sigma_{\mathrm{max}}(E)}{\sigma_{\mathrm{min}}(E)}$ and $M=n \sqrt{n} K (c_0 C  \log(ed) d^2 K \sigma(E))^{2n-2}$, we have
$$ \mathrm{Prob} \left( L(P) \leq  \frac{sK\sqrt{m}}{tM} \right)  \leq 
\exp(1- a_3u^2 n \log(ed)) + \left(\frac{s}{t}\right)^{\frac{1}{2}} u^{n-1}.$$
\noindent Using Lemma \ref{AG} and the assumption that $m \geq n \log(ed)$ 
we then obtain  
$$  \mathrm{Prob}\left( \tilde{\kappa}(P) \geq t M \right) \leq 
\exp(1-s^2 n \log(ed)^2) + \exp(1- a_3u^2 n \log(ed)) + 
\left(\frac{s}{t}\right)^{\frac{1}{2}} u^{n-1}.  $$
\noindent If $t \leq e^{2n\log(ed)}$ then setting $s=u=1$ gives the desired 
inequality.  If $t \geq e^{2n\log(ed)}$ then we  
set $s=u^2=\frac{\log(t)}{2\tilde{a}n\log(ed)}$, where $\tilde{a} > a_3 >0$ 
is a constant greater than $1$. We then obtain  
$$  \mathrm{Prob}\left( \tilde{\kappa}(P) \geq t M \right) \leq  
\exp\left(2- \frac{1}{2}\log(t)\right) +  
\left(\frac{\log(t)}{2n\log(ed)}\right)^{\frac{n}{2}} \frac{1}{\sqrt{t}}.  $$
\noindent Observe that  $\exp\left(2- \frac{1}{2}\log t\right) = 
\frac{e^2}{\sqrt{t}}$. So we have 
$  \mathrm{Prob}\left( \tilde{\kappa}(P) \geq t M \right) \leq   
\left(\frac{\log(t)}{2n\log(ed)}\right)^{\frac{n}{2}} \frac{e^2+1}{\sqrt{t}}$. 
To finalize our proof we need to check if $ t M \min \{ d^{-8} , 
n^{-1} \} \geq s K \sqrt{m}$. So we check the following: 
$$  t  K n \sqrt{m}  (c_0 C  \log(ed) d^2 K \sigma(E))^{2n-2} \min \{ d^{-8} , n^{-1} \}  \stackrel{?}{\geq}  \frac{\log(t)}{2 n\log(ed)} K \sqrt{m}.$$
For $n \geq 3$ we have $( d^2 \log(ed) )^{2n-2} > d^8$. 
Since $ K c_0 \geq \frac{1}{4}$, $C \geq 4$, and $\sigma(E) \geq 1$, we have
 $$ (c_0 C \log(ed) d^2 K \sigma(E))^{2n-2} > d^8 .$$
\noindent  Hence, it suffices to check if $t \geq \frac{\log(t)}{2 n\log(ed)}$, which is clear. \\

We would like to complete the proof of Theorem \ref{intro1} as it was stated in the introduction, for which the following easy observation suffices.
\begin{lem} \label{1}
For $t \geq e^{2n\log(ed)}$, we have $\left(\frac{\log(t)}{2n\log(ed)}
\right)^{\frac{n}{2}} \leq t^{\frac{1}{4\log(ed)}}$.
\end{lem}
\begin{proof}
Let $t = x e^{2n\log(ed)}$ where $x \geq 1$. Then 
$$\left(\frac{\log(t)}{2n\log(ed)}\right)^{\frac{n}{2}}  = \left(1 + 
\frac{\log(x)}{2n\log(ed)}\right)^{\frac{n}{2}} \leq e^{\frac{\log(x)}{4\log(ed)}} = x^{\frac{1}{4\log(ed)}} $$
\noindent Since $x \leq t$,  we are done.
\end{proof}
We now state the resulting bounds on the expectation of the condition number. 
\begin{cor} \label{2}
Under the assumptions of Theorem \ref{condition},  
$0 < q < \frac{1}{2} - \frac{1}{2\log(ed)}$ implies that 
$\mathbb{E} (\tilde{\kappa}(P)^q)  \leq  M^q ( 1 + 4q\log(ed) )$.  
In particular, $\mathbb{E} \log(\tilde{\kappa}(P)) \leq 1+ \log M$. 
\end{cor}
\begin{proof} 
Observe that 
\[ \mathbb{E} (\tilde{\kappa}(P)^q) = M^q + q M^q \int_{t=1}^{\infty} \mathbb{P}\{ \tilde{\kappa}(P) \geq t M\} \; t^{q-1} \; dt.  \] 
For $t \geq e^{2n\log(ed)}$, we have 
\[  \mathbb{P}\{ \tilde{\kappa}(P) \geq t M\} t^{q-1}  \leq t^{q+\frac{1}{4\log(ed)}-\frac{3}{2}} \leq t^{-1-\frac{1}{4\log(ed)}}  .\]
For $t \leq e^{2n\log(ed)}$ we have even stronger tail bounds: 
\[ \mathbb{E} (\tilde{\kappa}(P)^q)  \leq M^{q} \left( 1 + q \int_{t=1}^{\infty} t^{-1-\frac{1}{4\log(ed)}} \; dt  \right) .\]
This proves the first claim. The second claim follows by sending $q \rightarrow 0$ and using Jensen's inequality.
\end{proof}

\section{Proof of Theorem \ref{intro2}}
Let $E_i \subseteq H_{d_i}$ be non-degenerate linear spaces, and let $E=(E_1,\ldots, E_{n-1})$. 
Suppose $Q \in E$ is a fixed polynomial system. Let $g_i \in E_i$ be independent random elements of 
$E_i$ that satisfy the Centering Property,  the Sub-Gaussian Property with constant $K$, and the Anti-Concentration 
Property with constant $c_0$, each with respect to the Bombieri-Weyl inner  product. 
Let $G:=(g_1,\ldots,g_{n-1})$ be the corresponding polynomial system. 
We define random perturbation of $Q$ as follows: $P:=Q+G$. We will use this 
notation for $P$, $Q$ and $G$ for the rest of this section.
\begin{lem}\label{soperatornorm}
Let $Q \in E$ be a polynomial system, let $G$ be a random polynomial system 
in $E$ that satisfies the Centering, sub-Gaussian, and Anti-Concentration 
hypotheses, and let $P=Q+G$. Then we have 
$$ \mathrm{Prob} \left( \norm{P}_{\infty} \geq s \sigma_{\mathrm{max}}(E) \sqrt{n} \log(ed) + \norm{Q}_{\infty} \right) \leq \exp\left(1- \frac{a_3s^2 n \log(ed)}{K^2} 
\right)$$
\noindent where $a_3$ is an absolute constant.
\end{lem}

\begin{proof}
The triangle inequality implies $\norm{P}_{\infty} \leq 
\norm{Q}_{\infty} + \norm{G}_{\infty}$. We complete the proof by 
using Lemma \ref{operatornorm} for the random system $G$.
\end{proof}

\begin{lem} \label{ssmallball}
Let $Q \in E$ be a polynomial system, let $G$ be a random polynomial system 
in $E$ that satisfies the Centering, sub-Gaussian, and Anti-Concentration 
hypotheses, and let $P=Q+G$. Then, for all $\varepsilon > 0$, and for any 
$w \in S^{n-1}$ we have
$$ \mathrm{Prob} \left( \norm{ P(w) }_2 \leq \varepsilon \sigma_{\mathrm{min}}(E) \sqrt{n-1} \right) \leq  
( a_2 c_0 \varepsilon)^{n-1} $$
\noindent where $a_2$ is an absolute constant.
\end{lem}
\begin{proof}
By the Anti-Concentration Property, for all $1 \leq i \leq n-1$, we have
$$ \mathrm{Prob}  \{ \abs{ g_i(w) + q_i(w) } \leq c_0 \varepsilon \sigma_{\mathrm{min}}(E_i)  \} \leq c_0 \varepsilon $$
\noindent We then use Lemma \ref{smallballtensor} with the random variables 
$g_i(w) + q_i(w)$.
\end{proof}

\begin{lem} \label{sAG}
Let $Q \in E$ be a polynomial system, let $G$ be a random polynomial system 
in $E$ that satisfies the Centering, sub-Gaussian, and Anti-Concentration 
hypotheses, and let $P=Q+G$. Then for all $t \geq 1$, we have
$$ \mathrm{Prob}\left(  \|P\|_W \geq t  K \sqrt{\dim(E)}  + \norm{Q}_W \right) \leq \exp( 1 - t^{2} m ).  $$  
\end{lem}

\begin{proof}
For all $1 \leq i \leq n-1$, by triangle inequality $\norm{p_i}_W \leq \norm{q_i}_W + \norm{g_{C_i}}_W$. So using the first claim of Lemma \ref{AG} gives
$$ \mathrm{Prob}\left(  \| p_i \|_W \geq t  \sqrt{\dim(E_i)}  + \norm{q_i}_W \right) \leq \exp\left( 1 -\frac{t^{2} \dim(E_i)}{K^2}\right)  $$
Note that $\|P\|_W = \max_{\norm{w}_2=1} abs{\langle w , (p_1,\ldots,p_{n-1})} $. So proceding as in the proof of Lemma \ref{systembasic} completes the proof.
\end{proof}

\begin{thm} \label{sL-theorem}
Let $Q \in E$ be a polynomial system, let $G$ be a random polynomial system 
in $E$ that satisfies the Centering, sub-Gaussian, and Anti-Concentration 
hypotheses, and let $P=Q+G$. Now let $\gamma \geq 1 $, $d := \max_{i} d_i$, 
and assume $\alpha  \leq \min\{ d^{-8} , n^{-1} \} $. Then 
$$ \mathrm{Prob} ( L(P)\leq \alpha ) \leq  \mathrm{Prob} \left(\|P\|_{\infty} 
\geq \gamma\right)  + c \alpha^{\frac{1}{2}} \sqrt{n} \left( \frac{c_0 d^2 
\gamma C }{ \sigma_{\mathrm{min}}(E) \sqrt{n}}\right)^{n-1} $$ 
\noindent where $C$ is a universal constant. 
\end{thm}

The proof of Theorem \ref{sL-theorem} is identical to Theorem 
\ref{L-theorem}, so we skip it. Now we are ready to state main theorem of 
this section.

\begin{thm} \label{smooth}
Let $Q \in E$ be a polynomial system, let $G$ be a random polynomial system in $E$ that satisfies the Centering, sub-Gaussian, and Anti-Concentration 
hypotheses, and let $P=Q+G$.  Also let $d := \max_{i} d_i$,  and set 
$$M= n K \sqrt{\dim(E)} \left( c_0 d^2 C K \log(ed) \sigma(E) \right)^{2n-2} 
\left( 1 + \frac{\norm{Q}_{W}}{ \sqrt{n} K \log(ed)} \right)^{2n-1}  $$
\noindent where $C \geq 4$ is a universal constant. Assume also that 
$\dim(E)  \geq n \log(ed)^2$ and $n\geq 3$. Then \\  
\mbox{}\hfill $\mathrm{Prob}(\tilde{\kappa}(P) \geq t M) \leq  \begin{cases}
  \frac{3}{\sqrt{t}} &; \mbox{if } 1\leq t\leq e^{ 2 n \log{(ed)}} \\
  \frac{e^2+1}{\sqrt{t}} \left( \frac{{\log{t}}}{2 n \log{(ed)}}\right)^{\frac{n}{2}}  &; \mbox{if } e^{ 2n\log{(ed)} } \leq t.  \\
   \end{cases}$\hfill\mbox{}\\ 
\end{thm}

\begin{proof}
We need a quick observation before we start our proof: For any $Q \in E$ and 
$w \in S^{n-1}$, we have $\norm{Q(w)}_2^2 \leq \sum_{i=1}^{n-1} 
\norm{q_i}_W^2 \sigma_{\mathrm{max}}(E_i)^2 \leq \norm{Q}_W^2 
\sigma_{\mathrm{max}}(E)^2$. So we have 
$$\norm{Q}_{\infty} \leq \norm{Q}_{W} \sigma_{\mathrm{max}}(E) .$$ 
\noindent Using this upper bound on $\norm{Q}_{\infty}$ and the assumption 
that $\dim(E) \geq n \log(ed)^2$, we deduce\\  
\scalebox{.97}[1]{$ M \geq n K \sqrt{\dim(E)} \left( c_0 d^2 C K \log(ed) 
\sigma(E) \right)^{2n-2}\left(1 + \frac{\norm{Q}_W}{n K \sqrt{\dim(E)}}
\right) \left( 1 + \frac{\norm{Q}_{\infty}}{\sqrt{n} \log(ed) K 
\sigma_{\mathrm{max}}(E)}\right)^{2n-2}$.}\\ 
\noindent We will use this lower bound on $M$ later in our proof. Now let 
$m=\dim(E)$. We start our proof with the following observation: \\ 
\scalebox{.99}[1]{$\mathrm{Prob}\left( \tilde{\kappa}(P) \geq t M  \right) 
\leq  \mathrm{Prob}\left( \norm{P}_W \geq s K \sqrt{m} + \norm{Q}_W \right) 
+  \mathrm{Prob} \left( L(P) \leq  \frac{s K \sqrt{m} +  \norm{Q}_W}{tM} 
\right)$.}\\ 
\noindent Lemma \ref{sAG} states that  
$$ \mathrm{Prob}\left( \norm{P}_W \geq s K \sqrt{m} + \norm{Q}_W \right) \leq 
\exp( 1 - s^{2} m)  .$$
\noindent Theorem \ref{sL-theorem} states that for $ \frac{s K \sqrt{m} +  \norm{Q}_W}{tM} \leq \min\{ d^{-8} , n^{-1} \}$ we have\\ 
\scalebox{1}[1]{$\mathrm{Prob} \left( L(P) \leq  \frac{s K \sqrt{m} 
+  \norm{Q}_W}{tM} \right)    \leq     \mathrm{Prob} \left( \|P\|_{\infty} 
\geq \gamma \right)  + c ( \frac{s K \sqrt{m} +  
\norm{Q}_W}{tM})^{\frac{1}{2}} \sqrt{n} \left(\frac{c_0 d^2 \gamma C}
{\sigma_{\mathrm{min}}(E) \sqrt{n}}\right)^{n-1}$.}\\ 
\noindent We set $\gamma= u \sigma_{\mathrm{max}}(E) \sqrt{n} \log(ed) K + \norm{Q}_{\infty}  $. From Lemma \ref{soperatornorm}, we have
$$ \mathrm{Prob} \left( \norm{P}_{\infty} \geq u \sigma_{\mathrm{max}}(E) \sqrt{n} \log(ed) K + \norm{Q}_{\infty}  \right) \leq 
\exp(1- a_3u^2 n \log(ed)).$$  
\noindent We also have
$$ \left( \frac{c_0 d^2 \gamma C}{ \sigma_{\mathrm{min}}(E) \sqrt{n}}\right)^{n-1}  = \left( c_0 u d^2 C K \log(ed) \sigma(E)  \right)^{n-1}\left(1 + 
\frac{\norm{Q}_{\infty}}{u \sqrt{n} \log(ed) K \sigma_{\mathrm{max}}(E)} 
\right)^{n-1}.$$
\noindent Using $u \geq 1$, $s \geq 1$, $ m \geq n \log(ed)^2$, and the lower obtained on $M$, we obtain  
$$  \mathrm{Prob}\left( \tilde{\kappa}(P) \geq t M \right) \leq 
\exp(1-s^2 n \log(ed)) + \exp(1- a_3u^2 n \log(ed)) + \left(\frac{s}{t}
\right)^{\frac{1}{2}} u^{n-1}  $$ 
\noindent The rest of the proof is identical to the proof of Theorem \ref{condition}.
\end{proof}

\section{Proof of Theorem \ref{intro3}}
Define a random polynomial system $F_{\varepsilon}=Q+G$ where $G$ is Gaussian random polynomial system with $K=\frac{\varepsilon \norm{Q}_W}{\sqrt{n}\log(ed)}$ and $c_0 K=\frac{1}{\sqrt{2\pi}}$. Using Lemma \ref{AG} with $t=1$, we have with probability at least $1-\exp(1-\dim(E))$ that
$$ \norm{F_{\varepsilon}-Q}_{W} = \norm{G}_W \leq \frac{\varepsilon \norm{Q}_W \sqrt{\dim(E)}}{\sqrt{n}\log(ed)}. $$
\noindent For the condition estimate we will use Theorem \ref{smooth}: 
First note that with $K=\frac{\varepsilon \norm{Q}_W}{\sqrt{n}\log(ed)}$ and 
$c_0 K=\frac{1}{\sqrt{2\pi}}$, the quantity $M$ in Theorem \ref{smooth} is 
the following:  
$$ M= \frac{\varepsilon \sqrt{n} \sqrt{\dim(E)}}{\log(ed)} \left( \frac{d^2 
C \log(ed) \sigma(E)}{\sqrt{2\pi}} \right)^{2n-2} \left(1 + 
\frac{1}{ \varepsilon}\right)^{2n-1}. $$ 
\noindent So we have $M \leq 2 \sqrt{n} \sqrt{\dim(E)} 
\left(\frac{2}{\varepsilon}\right)^{2n-2} \left( \frac{d^2 C \log(ed) \sigma(E)}{\sqrt{2\pi}} \right)^{2n-2} $. Using Theorem \ref{smooth} with $t=36$ 
we deduce that with probability greater than $\frac{1}{2}$ we have
$$ \tilde{\kappa}(F_{\varepsilon}) \leq 2 \sqrt{n} \sqrt{\dim(E)}  \left( \frac{ d^2 C \log(ed) \sigma(E)}{\varepsilon} \right)^{2n-2}. $$
\noindent Since the union of the complement of these two events has measure 
less than\linebreak  
$\frac{1}{2} + \exp(1-\dim(E))$, their intersection has positive 
measure, and the proof is completed. \qed 

\begin{rem} \label{TT}
The proof of Theorem \ref{sL-theorem} actually works for
$$ M = n K \sqrt{\dim(E)} \left( c_0 d^2 C K \log(ed) \sigma(E) \right)^{2n-2}(1 + \norm{Q}_W) \left( 1 + \frac{\norm{Q}_{\infty}}{\sqrt{n} \log(ed) K 
\sigma_{\mathrm{max}}(E)} \right)^{2n-2}, $$
\noindent which is often much more smaller than the $M$ used in the theorem 
statement. \dia 
\end{rem}

\section{Appendix: The Dispersion Constants of Random Subspaces of 
Polynomial Systems}
Here we address the question how big the dispersion constant is for 
a ``typical'' low-dimensional linear space. Imagine we have fixed a dimension 
$m \sim n \log d$ and wish to consider subspaces of dimension $m$ inside 
$H_d$ (the vector space homogenous polynomials of degree $d$).  How does the 
dispersion constant vary among these subspaces?  We know that some of these 
subspaces will be degenerate and have infinite dispersion constant. Can we 
argue that high dispersion constants are rare?

To address this problem, we represent the space of 
$m$-dimensional linear subspaces of $H_d$ by the {\em Grassmannian variety},  
$\mathrm{Gr}(m,\dim(H_d))$, which comes equipped with a Haar measure. We will 
analyze the Haar measure of the set of subspaces in $\mathrm{Gr}(m,\dim(H_d))$ 
that yield high dispersion constant (see Corollary \ref{dispersion} below).

We will first need to introduce the following notion from high-dimensional 
probability. 
\begin{dfn}[Gaussian Complexity] \label{gcomplexity}
Let $X \subseteq \mathbb{R}^n$ be a set, then the {\em Gaussian complexity} 
of $X$ denoted by $\gamma(X)$ is defined as follows:

$$ \gamma(X) := \mathbb{E} \sup_{x \in X} \abs{\langle G , x \rangle} $$

\noindent where $G$ is distributed according to standard normal distribution $\mathcal{N}(0,\mathbb{I})$ on $\mathbb{R}^n$. \dia 
\end{dfn}

The use of the term {\em complexity} in definition \ref{gcomplexity} might 
look  unorthodox to readers with a computational complexity theory background. 
The rationale behind this standard terminology in high dimensional probability  
is that the Gaussian complexity of a set $X$ is known to control the 
complexity of stochastic processes indexed on the set $X$ (see, e.g., 
\cite{Ve-1}). 

A corollary of Lemma \ref{basic} and Lemma \ref{net-norm} is the following.
\begin{cor} [Gaussian Complexity of the Veronese Embedding] 
\label{gaussian} 
Let $H_d$ be the vector space of degree $d$ homogenous polynomials in $n$ 
variables. Let $u_i \; i=1,\ldots, \binom{n+d-1}{d}$ be an orthonormal basis for the 
vector space $H_d$ with respect to the Bombieri-Weyl norm. For every $v \in S^{n-1}$, we define 
the following polynomial $q_v$: 
\[ q_v(x) := \sum_{i} u_{i}(v) u_i(x) \]
and the following set created out of $q_v$:
\[ B_{d} := \{ q_v : v \in S^{n-1} \}  \] 
Then we have  $\gamma(B_d) \leq c \sqrt{n} \log(ed)$ for a universal constant 
$c$. 
\end{cor}

\noindent 
{\bf Proof of Corollary \ref{gaussian}:} 
We need to consider a Gaussian element $G$ in the vector space $H_d$.  Note that for $G \sim \mathcal{N}(0,\mathbb{I})$ in $H_d$ we have 
$\left\langle G , \sqrt{\binom{d}{\alpha}} x^{\alpha} \right\rangle_W 
\sim \mathcal{N}(0,1)$ since $\sqrt{\binom{d}{\alpha}} x^{\alpha}$ is an 
orthonormal basis with respect to the Weyl-Bombieri inner product. This means 
Gaussian elements of $H_d$ are included in our model of randomness for the 
special case $K=1$.  Since $\sigma_{\mathrm{max}}(H_d)=1$, Lemma \ref{basic} 
gives us the following estimate for pointwise evaluations of the Gaussian 
element $G \sim \mathcal{N}(0,\mathbb{I})$ in $H_d$: 
$$ \mathrm{Prob} \{ \abs{G(v)} \geq t  \} \leq \exp\left(1-\frac{t^2}{2}
\right).$$

\noindent 
Note that $\norm{G}_{\infty}= \max_{v \in S^{n-1}} \abs{G(v)} = \max_{q_v \in B_d} \abs{\langle G ,q_v \rangle}$. So to estimate Gaussian complexity of the Veronese embedding $B_d$, we need to estimate $\mathbb{E} \norm{G}_{\infty}$. Let $\mathcal{N}$ be a $\delta$-net on the sphere $S^{n-1}$. Using a union bound, we 
then have
$$ \mathrm{Prob} \{ \max_{v \in \mathcal{N}} \abs{G(v)} \geq t  \} \leq \abs{\mathcal{N}} \exp\left(1-\frac{t^2}{2}\right) .$$ 
Setting $\delta=\frac{1}{d}$ and using Lemma \ref{net-norm} for $t \geq a_1 \sqrt{n} \log(ed)$ then gives the following: 
$$ \mathrm{Prob} \{ \norm{G}_{\infty} \geq a_1 t \sqrt{n} \log(ed) \} \leq  \abs{\mathcal{N}} \exp\left(1-\frac{a_1^2 t^2 n \log(ed)}{2}\right) $$

\noindent 
It is known that $\abs{\mathcal{N}} \leq \exp(a_0 n \log d)$. So  
we have 
$$ \abs{\mathcal{N}} \exp\left(1-\frac{a_1^2 t^2 n \log(ed)}{2}\right) 
\leq \exp(1 - a_2 t^2 n \log(ed) ) $$ 

\noindent for some constant $a_2$. So  
$\mathrm{Prob} \{ \norm{G}_{\infty} \geq a_1 t \sqrt{n} \log(ed) \} \leq 
\exp(1 - a_2 t^2 n \log(ed) )$. 
Using this inequality one can routinely derive the estimate for $\mathbb{E} 
\norm{G}_{\infty}$. \qed 

Since Talagrand proved his celebrated ``majorizing measure theorem'' 
(see \cite{Tal}) it has been observed that for a set $X$ and a random 
$k\times n$ sub-Gaussian matrix $A$, the deviation 
$ \sup_{x\in X} \left | \| Ax\|_{2} - \mathbb E \| Ax\|_{2}  \right |$  is 
controlled by the Gaussian complexity $\gamma(X)$. 
We will use a variant  established in \cite{KLM} but not stated 
explicitly:   
\begin{thm} \label{Haar} Let $F$ be a random $m$ dimensional subspace of 
$\mathbb{R}^n$ drawn from Haar measure on $\mathrm{Gr}(m,n)$, and let $P_F$ be 
orthogonal projection map on $F$.  Let $X \subseteq \mathbb{R}^n$ be a set. 
Then there is a universal constant $C$ such that
$$ \sup_{x \in X} \abs{\sqrt{n} \norm{P_F(x)} - \sqrt{m} \norm{x} } \leq C t 
\gamma(X), \ t\geq 1 $$

\noindent with probability greater than $1-e^{-t^2}$.
\end{thm}

\noindent 
There is a series of papers 
that established several variants of the preceding two deviation bounds 
--- mainly in \cite{Sch-1,KLM} and, more recently in \cite{Dir,LMPR}. 
Vershynin devoted the 9th chapter of his recent book \cite{Ve-1} on these 
results and their applications. Theorem \ref{Haar} follows easily upon 
combining some statements and exercises from \cite[Ch.\ 9]{Ve-1}. 
We include a sketch of the proof below for the interested reader. 

\begin{proof}[Proof of Theorem \ref{Haar}]
Let $ x\in X$ and consider the random
process $ W_{x}:= \sqrt{ n} \|P x\|- \sqrt{m} \|x\|$. By [\cite{KLM}, Lemma
4.2, \cite{Sch-1}] we have that $W_{x} $ is a subgaussian process in $X$, i.e.,
$$ \mathbb P \left( \abs{ \| Px\| - \| P y \| } \geq s \| x- y \| \right) \leq 2 e^{ - c s^{2}} , s>0$$
where $ c>0$ is an absolute constant and $x, y\in X$, or equivalently
$$ \big\|  \| Px\| - \| P y \| \big\|_{\psi_{2}} \leq c \| x- y \| .$$
  In [\cite{KLM}, Lemma 4.2], the above inequality is stated for $ x, y\in S^{n-1}$. To extend it for every $ x, y$ is straightforward,  and we explain the idea below ( see e.g. proof of Lemma 9.1.4 in \cite{Ve-1} or \cite{LMPR} for details). By scaling, without loss of generality we may assume that $ \| x\|=1, \| y\|\geq 1$. Set $ \tilde{y} =\frac{ y}{ \|y\|} $. Note that
  $$ \| W_{y}- W_{\bar{y}} \|= \| y - \bar{y} \| \| W_{\bar{y}} \|_{\psi_{2} } \leq C \| y - \bar{y} \| $$
  for a universal constant $C$. Using all the above and the triangle inequality we get that
$$ \| W_{x}- W_{y} \|_{\psi_{2}} \leq C \left( \| x- \bar{y} \| + \| y- \bar{y} \| \right) \leq \sqrt{2} C \| x- y \| . $$
Now that we have established that $ W_{x} $ for $ x\in X$ is a subgaussian process we may apply [\cite{Dir} Theorem 3.2] or [\cite{Tal} Theorem 2.2.27] to conclude the proof. For example the latter states that
$$ \mathbb P \left( \sup_{x,y\in X} \abs{W_{x}- W_{y} } \geq C \left( \gamma_{2}(X, \| \cdot \|)  + s {\rm diam}(X) \right)\right) \leq 2e^{- s^{2}} . $$
Here $ {\rm diam} (X) := \max_{x,y\in X} \| x- y \|_{2}$ and $ \gamma_{2}$ is Talagrand's functional (see \cite{Ve-1}, Definition of 8.5.1 for details).
By Talagrand's majorizing measure theorem (see e.g. \cite{Ve-1}, Theorem 8.6.1)  it is known that $\gamma_{2}(X, \| \cdot \|) \simeq \gamma(X) $. Using the triangle inequality  and the fact that ${\rm diam(X) } \leq 2 \gamma(X) $, we conclude that
$$ \mathbb P \left( \sup_{x\in X} \abs{ W_{x}} \geq c s \gamma(X) \right) \leq 2 e^{ - s^{2}} , \ s \geq 1.  $$
\end{proof}

A simple consequence of Theorem \ref{Haar} is the 
following estimate on the dispersion constant of a random subspace of 
polynomial systems: 
\begin{cor}\label{dispersion}
Let $F$ be a random $m$ dimensional subspace of $H_d$ drawn from the Haar measure on $\mathrm{Gr}(m,\dim(H_d))$, where $m \geq 16Cn \log(ed)^2$. Then 
$$ \sigma(F) \leq \frac{\sqrt{m} + Ct\sqrt{n}\log(ed)}{\sqrt{m} - Ct\sqrt{n}\log(ed)}  $$ 

\noindent with probability greater than $1- e^{-t^2}$, where $C$ is the 
absolute constant from Theorem \ref{Haar}. 
\end{cor}

\noindent 
{\bf Proof of Corollary \ref{dispersion}:}  
Since $\norm{q_v}_W=1$ for all $v \in S^{n-1}$,  applying Theorem \ref{Haar} 
to the set $B_d$ implies that 
$$ \sup_{x \in B_d} \abs{\binom{n+d-1}{d}^{\frac{1}{2}} \norm{\Pi_F(x)} - \sqrt{m} } \leq C t \sqrt{n}\log(ed) $$
\noindent with probability greater than $1-e^{-t^2}$ for all $t \geq 1$. Since $\sigma_{min}(F)=\min_{x \in B_d} \norm{\Pi_F(x)}$ and $\sigma_{max}(F)=\max_{x \in B_d} \norm{\Pi_F(x)}$, we have 
$$    \frac{\sqrt{m}- C t \sqrt{n}\log(ed)}{\binom{n+d-1}{d}^{\frac{1}{2}}} \leq \sigma_{\mathrm{min}}(F) \leq \sigma_{\mathrm{max}}(F) \leq \frac{\sqrt{m}+C t \sqrt{n}\log(ed)}{\binom{n+d-1}{d}^{\frac{1}{2}}} $$
\noindent with probability greater than $1-e^{-t^2}$.  \qed


\begin{thebibliography}{9}

\bibitem{BK} {\rm Franck Barthe and Alexander Koldobsky}, 
{\sl ``Extremal slabs in the cube and the Laplace transform,''} Adv.\ Math.\ 
174 (2003),  pp.\ 89--114. 


\bibitem{beltran} Carlos Beltr\'an and Luis-Miguel Pardo, \emph{``Smale's 
17th problem: Average polynomial time to compute affine and projective 
solutions,''} Journal of the American Mathematical Society \textbf{22} 
(2009), pp.\ 363--385. 

\bibitem{bcss} {\rm Lenore Blum, Felipe Cucker, Mike Shub, and Steve Smale}, 
{\sl Complexity and Real Computation}, Springer-Verlag, 1998. 

\bibitem{Bou} {\rm  Jean Bourgain}, {\sl ``On the isotropy-constant problem 
for $\psi_2$-bodies,''} in {\em Geometric Aspects of Functional Analysis}, 
Lecture Notes in Mathematics, vol.\ 1807, pp.\ 114--121, Springer Berlin 
Heidielberg, 2003.

\bibitem{derand} {\rm Peter B\"u{}rgisser and Felipe Cucker,} 
{\sl ``On a problem posed by Steve Smale,''} Annals of Mathematics, 
pp.\ 1785--1836, Vol.\ 174 (2011), no.\ 3.  
 
\bibitem{cond} {\rm Peter B\"u{}rgisser and Felipe Cucker,} {\sl Condition,} 
Grundlehren der mathematischen Wissenschaften, no.\ 349, Springer-Verlag, 
2013. 

\bibitem{homosemi} {\rm Peter B\"urgisser, Felipe Cucker, Pierre Lairez},  
{\sl Computing the homology of basic semialgebraic sets in weakly 
exponential time}, Journal of the ACM (JACM) 66, 2018, pg 1--30 

\bibitem{homosemi2}  {\rm Peter B{\"u}rgisser, Felipe Cucker , Josu{\'e} Tonelli-Cueto} , {\sl Computing the homology of semialgebraic sets. I: Lax formulas},
Foundations of Computational Mathematics, 2018

\bibitem{homosemi3} {\rm Peter B{\"u}rgisser, Felipe Cucker , Josu{\'e} Tonelli-Cueto} ,{\sl Computing the Homology of Semialgebraic Sets. II: General formulas}, arXiv preprint arXiv:1903.10710 

\bibitem{pv} {\rm Felipe  Cucker, Alperen A. Erg\"ur, Josu{\'e} Tonelli-Cueto}, {\sl Plantinga-Vegter algorithm takes average polynomial time.}, In Proceedings of the 2019 on International Symposium on Symbolic and Algebraic Computation, pp. 114-121. 2019.


\bibitem{pardo} {\rm D.\ Castro, Juan San Mart\'{\i}n, Luis M.\ Pardo}, 
{\sl ``Systems of Rational Polynomial Equations have Polynomial Size 
Approximate Zeros on the Average,''} Journal of Complexity 19 (2003), 
 pp.\ 161--209. 


\bibitem{cucker} Felipe Cucker, {\sl ``Approximate zeros and condition 
numbers,''} Journal of Complexity 15 (1999), no.\ 2, pp.\ 214--226. 

\bibitem{M1} {\rm Felipe Cucker, Teresa Krick, Gregorio Malajovich, and Mario 
Wschebor}, {\sl ``A numerical algorithm for zero counting I. Complexity and accuracy,''}  
J.\ Complexity 24 (2008), no.\ 5--6, pp.\ 582--605

\bibitem{M2} {\rm Felipe Cucker, Teresa Krick, Gregorio Malajovich, and Mario 
Wschebor}, {\sl ``A numerical algorithm for zero counting II. Distance to ill-posedness 
and smoothed analysis,''} J.\ Fixed Point Theory Appl.\ 6 (2009), no.\ 2, 
pp.\ 285--294.

\bibitem{M3} {\rm Felipe Cucker, Teresa Krick, Gregorio Malajovich, and Mario 
Wschebor}, {\sl ``A numerical algorithm for zero counting III: Randomization and 
condition,''}  Adv.\ in Appl.\ Math.\ 48 (2012), no.\ 1, pp.\ 215--248.

\bibitem{cks} Felipe Cucker; Teresa Krick; and Mike Shub, {\it ``Computing the Homology of Real Projective Sets,''} 
Foundations of Computational Mathematics, Foundations of Computational Mathematics, 2018, pg 929--970


\bibitem{DS1} {\rm Jean-Piere Dedieu, Mike Shub}, {\sl ``Newton's Method for 
Overdetermined Systems Of Equations,''} Math.\ Comp.\ 69 (2000), no.\ 231, 
pp.\ 1099--1115. 

\bibitem{demmel} {\rm James Demmel, Benjamin Diament, and Gregorio 
Malajovich},  {\sl ``On the Complexity of Computing Error Bounds,''} Found.\ Comput.\ Math.\ 
pp.\ 101--125 (2001). 

\bibitem{Dir} {\rm S. Dirksen},  {\sl Tail bounds via generic chaining}, 
Electronic J.\ Probab.\ 20 (2015), no.\ 53, 29 pp. 

\bibitem{EPR17} {\rm Alperen A.\ Erg\"ur, Grigoris Paouris, J.\ Maurice 
Rojas}, {\sl ``Probabilistic Condition Number Estimates For Real Polynomial 
Systems I: A Broader Family Of Distributions''}, Foundations of Computational 
Mathematics, to appear. DOI: 10.1007/s10208-018-9380-5. 


\bibitem{hoeffding} {\rm Wassily Hoeffding}, {\sl ``Probability inequalities 
for sums of bounded random variables,"} 
Journal of the American Statistical Association, 58 (301):13--30, 1963. 



\bibitem{kellog} {\rm O.\ D.\ Kellog}, {\sl ``On bounded polynomials in 
several variables,''} Mathematische Zeitschrift, December 1928, Volume 27, 
Issue 1, pp.\ 55--64. 

\bibitem{KM} {\rm Bo'az Klartag and Emanuel Milman}, {\sl ``Centroid bodies 
and the logarithmic Laplace Transform -- a unified approach,''} J.\ Func.\ 
Anal., 262(1):10--34, 2012.  

\bibitem{KlM} {\rm Bo'az Klartag and Shahar Mendelson}, {\sl Empirical processes and random projections}
, J. Functional Analysis, Vol. 225, no. 1 (2005) 229--245.

\bibitem{KLM} {\rm Bo'az Klartag and Shahar Mendelson}, {\sl Empirical 
processes and random projections}
, J. Funct. Anal.  Vol. 225, no. 1 (2005) 229--245.

\bibitem{kp} {\rm Alexander Koldobsky and Alain Pajor}, {\sl ``A Remark on Measures 
of Sections of $L_p$-balls,''} Geometric Aspects of Functional Analysis, 
Lecture Notes in Mathematics 2169, pp.\ 213--220, Springer-Verlag, 2017. 


\bibitem{kostlan} {\rm Eric Kostlan}, {\sl ``On the Distribution of Roots of 
Random Polynomials,''}  Ch.\ 38 (pp.\ 419--431) of {\em From Topology to 
Computation: Proceedings of Smalefest (M.\ W.\ Hirsch, J.\ E.\ Marsden, 
and M.\ Shub, eds.)}, Springer-Verlag, New York, 1993. 

\bibitem{lairez} {\rm Pierre Lairez}, {\sl ``A deterministic algorithm to 
compute approximate roots of polynomial systems in polynomial average time,''} 
Foundations of Computational Mathematics, DOI 10.1007/s10208-016-9319-7. 


\bibitem{ledoux} Michel Ledoux, {\it The Concentration of Measure 
Phenomemon,} Mathematical Surveys \& Monographs, Book 89, 
AMS Press, 2005. 

\bibitem{LMPR} {\rm C.\ Liaw, A.\ Mehrabian, Y.\ Plan, and R.\ Vershynin}, 
{\sl ``A simple tool for bounding the deviation of random matrices on 
geometric sets,''} Geometric Aspects of Functional Analysis: Israel Seminar 
(GAFA) 2014--2016, Lecture Notes in Mathematics 2169, Springer, 2017, 
pp.\ 277--299. 

\bibitem{LPP} {\rm G.\ Livshyts, Grigoris Paouris and P.\ Pivovarov}, 
{\sl  ``Sharp bounds for marginal densities of product measures,''} 
Israel Journal of Mathematics, Vol.\ 216, Issue 2, pp.\ 877--889. 



\bibitem{NZ} {\rm Assaf Naor and Artem Zvavitch}, {\sl ``Isomorphic embedding
of $\ell_p^n$, $1 < p < 2$, into $\ell_1^{(1+\varepsilon)n}$,''}
Israel J.\ Math.\ {\bf 122} (2001), pp.\ 371--380.

\bibitem{Ng} {\rm Hoi H.\ Nguyen}, {\sl ``On a condition number of general 
random polynomial systems,''} Mathematics of Computation (2016) 85, 
pp.\ 737--757


\bibitem{PP} {\rm Grigoris Paouris and Peter Pivovarov}, {\sl ``Randomized Isoperimetric Inequalities'', } IMA Volume "Discrete Structures: Analysis and Applications" (Springer)  


\bibitem{RV} {\rm Mark Rudelson and Roman Vershynin}, {\sl ``The 
Littlewood-Offord Problem and Invertibility of Random Matrices,''} 
Adv.\ Math.\ 218 (2008), no.\ 2, pp.\ 600--633.

\bibitem{RV-2} {\rm Mark Rudelson and Roman Vershynin}, {\sl ``The Smallest 
Singular Value of Rectangular Matrix,''} Communications on Pure and Applied Mathematics 
62 (2009), pp.\ 1707--1739. 

\bibitem{RV-1} {\rm Mark Rudelson and Roman Vershynin}, {\sl ``Small ball 
Probabilities for Linear Images of High-Dimensional Distributions,''} Int.\ 
Math.\ Res.\ Not.\ (2015), no.\ 19, pp.\ 9594--9617. 

\bibitem{Sch-1} {\rm G. Schechtman,} {\sl Two observations regarding embedding subsets of Euclidean spaces in normed spaces}, Adv. Math. 200 (2006), 125--135.

\bibitem{shafa} {\rm Igor R.\ Shafarevich}, {\sl Basic
Algebraic Geometry 1: Varieties in Projective Space}, 3rd edition, 
Springer-Verlag (2013).  

\bibitem{SS} {\rm Mike Shub and Steve Smale}, {\sl ``Complexity of Bezout's
Theorem I. Geometric Aspects,''} J.\ Amer.\ Math.\ Soc.\ 6 (1993), no.\ 2,
pp.\ 459--501.

\bibitem{Smale} {\rm Steve Smale}, {\sl ``Mathematical Problems for the next 
Century,''} The Mathematical Intelligencer, 1998, 20.2, pp.\ 7–-15. 

\bibitem{ST} {\rm D.\ A.\ Spielman and S.-H.\ Teng}, {\sl ``Smoothed analysis 
of algorithms,''}  Proc.\ Int.\ Congress Math. (Beijing, 2002), 
Volume I, 2002, pp.\ 597--606.



\bibitem{Tal} {\rm M. Talagrand},  {\sl The generic chaining. Upper and lower bounds of stochastic processes}. Springer Monographs in Mathematics. Springer-Verlag, Berlin, 2005.

\bibitem{Ve-1} {\rm Roman Vershynin }, {\sl ``High Dimensional Probability: An Introduction with Application in Data Science'',} Cambridge Series 
	in Statistical and Probabilistic Mathematics (Book 47), 
	Cambridge University Press, 2018. 

\bibitem{Ve-2} {\rm Roman Vershynin }, {\sl ``Four Lectures on Probabilistic Methods for Data Science'',} available at https://www.math.uci.edu/~rvershyn/papers/four-lectures-probability-data.pdf

\bibitem{V} {\rm Roman Vershynin}, {\sl ``Introduction to the Non-Asymptotic 
Analysis of Random Matrices,''} Compressed sensing, pp.\ 210--268, Cambridge 
Univ.\ Press, Cambridge, 2012.


\end{thebibliography}
\end{document}